\documentclass[review,3p,square]{elsarticle}

\usepackage{amssymb}
\usepackage{amsmath}
\usepackage{amsthm}
\usepackage{thmtools}
\usepackage{stmaryrd}
\usepackage{bbm}
\usepackage{enumitem}

\usepackage{array}
\usepackage{multirow}
\usepackage{makecell}
\usepackage{longtable}
\usepackage{caption}

\usepackage{fancyhdr}
\usepackage{hyperref}
\hypersetup{%
  pdftitle={BSDEs},%
  pdfsubject={BSDEs},%
  pdfauthor={ShengJun FAN, Ying Hu},%
  pdfkeywords={BSDEs},%
  pdfstartview=FitH,%
  CJKbookmarks=true,%
  bookmarksnumbered=true,%
  bookmarksopen=true,%
  colorlinks=true, linkcolor=blue, urlcolor=blue, citecolor=blue, %
}

\usepackage{cleveref}
\crefname{thm}{Theorem}{Theorems}
\crefname{pro}{Proposition}{Propositions}
\crefname{lem}{Lemma}{Lemmas}
\crefname{rmk}{Remark}{Remarks}
\crefname{cor}{Corollary}{Corollaries}
\crefname{defn}{Definition}{Definitions}
\crefname{ex}{Example}{Examples}
\crefname{section}{Section}{Sections}
\crefname{subsection}{Subsection}{Subsections}

\newcommand{\eps}{\varepsilon}

\newcommand{\To}{\rightarrow}
\newcommand{\as}{{\rm d}\mathbb{P}\times{\rm d} t-a.e.}
\newcommand{\ass}{{\rm d}\mathbb{P}\times{\rm d} s-a.e.}
\newcommand{\ps}{\mathbb{P}-a.s.}

\newcommand{\esup}{{\rm ess} \sup}

\newcommand{\F}{\mathcal{F}}
\newcommand{\E}{\mathbb{E}}

\newcommand{\s}{\mathcal{S}}

\newcommand{\mcal}{\mathcal{M}}
\newcommand{\ecal}{\mathcal{E}}
\newcommand{\acal}{\mathcal{A}}

\newcommand{\T}{[0,T]}

\newcommand{\Ln}{\mathcal{I\!L}_n}

\newcommand{\p}{{\mathbb P}}
\newcommand{\Q}{{\mathbb Q}}
\newcommand{\R}{{\mathbb R}}
\newcommand{\RE}{\forall}

\newcommand {\Dis}{\displaystyle}

\newtheorem{thm}{Theorem}[section]

\newtheorem{pro}[thm]{Proposition}
\newtheorem{rmk}[thm]{Remark}

\newtheorem{defn}[thm]{Definition}
\newtheorem{ex}[thm]{Example}

\journal{arXiv}
\begin{document}
\begin{frontmatter}

\title{{1D nonlinear backward stochastic differential equations:
\\ a unified theory and applications}
\tnoteref{found}}
\tnotetext[found]{This work is supported by National Natural Science Foundation of China (Nos. 12171471 and 11631004), by Key Laboratory of Mathematics for Nonlinear Sciences (Fudan University), Ministry of Education, Handan Road 220, Shanghai 200433, China; by Lebesgue Center of Mathematics ``Investissements d'avenir" program-ANR-11-LABX-0020-01, by CAESARS-ANR-15-CE05-0024 and by MFG-ANR-16-CE40-0015-01.
\vspace{0.2cm}}


\author[Fan]{Shengjun Fan} \ead{shengjunfan@cumt.edu.cn}
\author[Hu]{Ying Hu\corref{cor}} \ead{ying.hu@univ-rennes1.fr}
\author[Tang]{Shanjian Tang} \ead{sjtang@fudan.edu.cn} \vspace{-0.5cm}

\affiliation[Fan]{organization={School of Mathematics, China University of Mining and Technology},
            city={Xuzhou 221116},
            country={China}}

\affiliation[Hu]{organization={Univ. Rennes, CNRS, IRMAR-UMR6625},
            city={F-35000, Rennes},
            country={France}}

\affiliation[Tang]{organization={Department of Finance and Control Sciences, School of Mathematical Sciences, Fudan University},
            city={Shanghai 200433},
            country={China}}

\cortext[cor]{Corresponding author\vspace{0.2cm}}

\begin{abstract}
  Since the celebrated paper by El Karoui, Peng and Quenez \cite[Mathematical Finance, 7 (1997), 1--71]{ElKarouiPengQuenez1997MF}, backward stochastic differential equations have found wide applications in stochastic control, financial technology and machine learning.
In this paper, we present a comprehensive theory on the existence and uniqueness  of adapted solutions to a one-dimensional nonlinear backward stochastic differential equation (1D BSDE for short), and  assume that the generator $g$ has a unilateral linear or super-linear growth in the first unknown variable $y$,  and has an at most quadratic growth in the second unknown variable $z$. We develop a unified methodology, featured by the test function method and the a priori estimate technique,  to  establish several existence theorems and comparison theorems, which immediately yield corresponding existence and uniqueness results. We also  overview  relevant known results and give some practical applications of our theoretical results. Finally, we list some open problems  on the well-posedness of 1D BSDEs.\vspace{0.2cm}
\end{abstract}

\begin{keyword}
Backward stochastic differential equation \sep A unified theory \sep Feynman-Kac formula \\  \hspace*{1.8cm} Existence and uniqueness \sep Comparison theorem \sep  Open problems \sep  $g$-expectation.
\vspace{0.2cm}

\MSC[2010] 60H10\vspace{0.2cm}
\end{keyword}

\end{frontmatter}
\vspace{-0.4cm}

\section{Introduction}
\label{sec:1-Introduction}
\setcounter{equation}{0}

Fix a real $T>0$ and an integer $d\geq 1$. Let $(\Omega, \F, \mathbb{P})$ be a complete probability space equipped with augmented filtration $(\F_t)_{t\in\T}$ generated by a standard $d$-dimensional Brownian motion $(B_t)_{t\in\T}$, and assume that $\F_T=\F$. The equality or inequality between random elements is  understood in the sense of $\ps$ We consider the following one-dimensional backward stochastic differential equation (1D BSDE in short):
\begin{equation}\label{eq:1}
  Y_t=\xi+\int_t^T g(s,Y_s,Z_s){\rm d}s-\int_t^T Z_s\cdot {\rm d}B_s, \ \ t\in\T.
\end{equation}
Here,  $\xi$ is called the terminal value being an $\F_T$-measurable real random variable, the random field
$$
g(\omega, t, y, z):\Omega\times\T\times\R\times\R^d \to \R
$$
is called the generator of \eqref{eq:1}, which is $(\F_t)$-adapted for each $(y,z)$, and the pair of $(\F_t)$-adapted and $\R\times\R^d$-valued processes $(Y_t,Z_t)_{t\in\T}$ is called a solution of \eqref{eq:1} if $\ps$, $t\mapsto Y_t$ is continuous, $t\mapsto |g(t,Y_t,Z_t)|+|Z_t|^2$ is integrable, and \eqref{eq:1} is satisfied. Denote by BSDE$(\xi,g)$ the BSDE with the terminal condition $\xi$ and the generator $g$, which are called the parameters of the BSDE.

For  convenience of exposition, throughout the  paper,  let us always fix the constants $\alpha\in [1,2]$, $\beta\geq 0$, $\gamma>0$, $\delta\in [0,1]$, and $\lambda\in \R$,  and an $(\F_t)$-progressively measurable $\R_+$-valued stochastic process $(f_t)_{t\in \T}$. We assume that the generator $g$ satisfies $\as$,
\begin{equation}\label{eq:1.2}
 {\rm sgn}(y)g(\omega,t,y,z)\leq f_t(\omega)+\beta |y|(\ln(e+|y|))^\delta+\gamma |z|^\alpha(\ln(e+|z|))^\lambda
\end{equation}
for any $(y,z)\in \R\times\R^d$.
We usually say that $g$ has a unilateral linear growth in the state variable $y$ when $\delta=0$, and a unilateral super-linear growth in $y$ when $\delta\in (0,1]$. Furthermore, for the case of $\lambda=0$, we say that $g$ has a power sub-linear growth in the state variable $z$ when $\alpha\in (0,1)$, a linear growth in $z$ when $\alpha=1$, a sub-quadratic growth in $z$ when $\alpha\in (1,2)$, a quadratic growth in $z$ when $\alpha=2$, a super-quadratic growth in $z$ when $\alpha>2$, and for the case of $\alpha=1$ and $\lambda\neq 0$, we say that $g$ has a logarithmic sub-linear growth in $z$ when $\lambda<0$, and a logarithmic super-linear growth in $z$ when $\lambda>0$.

\subsection{Overview of relevant existing results\vspace{0.2cm}}

BSDEs were initially introduced by Bismut~\cite{Bismut1973JMAA,Bismut1976SICON,Bismut1978SIAMReview} in the linear case. General nonlinear BSDEs were first investigated by Pardoux and Peng~\cite{PardouxPeng1990SCL}, where an existence and uniqueness result was established on adapted solutions of multidimensional BSDEs with square-integrable parameters and uniformly Lipschitz continuous generators. Subsequently, BSDEs have received  an  extensive  attention due to its various connections to numerous topics such as stochastic control, financial technology, machine learning, partial differential equations (PDEs in short), nonlinear mathematical expectation and so on. In particular, the paper by El Karoui, Peng and Quenez \cite[Mathematical Finance, 7 (1997), 1--71]{ElKarouiPengQuenez1997MF} has seen a pervasive role of BSDEs in stochastic control and financial technology. The reader is also referred to for example \cite{Peng1990SICON,Peng1991Stochastics,Peng1992Stochastics,Peng1993AMO, PardouxTang1999PTRF,YongZhou1999,Kobylanski2000AP,KohlmannTang2002SPA,Tang2002RDMF,
KohlmannTang2003SICON,Tang2003SICON,HuImkeller2005AAP,HuZhou2005SICON,Jiang2005SPA,
Jiang2008AAP,Tevzadze2008SPA,DelbaenTang2010PTRF,HuZhou2012SICON,
PardouxRascanu2014Book,XingZitkovic2018AoP,
FanHuTangArXiv2024,TangZhouArXiv2024,FanWangYong2025JDE} for more details.
Note that recently, BSDEs of mean-field type have found wide applications in Mean Field Games, see, e.g. \cite{CD1,CD2,CDLL}.

In particular, a lot of  attentions have been paid on the well-posedness of adapted solutions of BSDEs under various growth and/or continuity  of the generator $g$ with respect to both unknown variables $(y,z)$ and various integrability of the parameters $(\xi, f)$. Generally speaking, these efforts fall into three different classes. The first one focuses on  $L^p$ solution   ($p\geq 1$)  of BSDEs. Relevant classical results are available in \cite{LepeltierSanMartin1997SPL,Pardoux1999Nonlinear,BriandCarmona2000IJSA,
BriandDelyonHu2003SPA,Jia2008CRA,Jia2008PHDThesis,FanJiangDavison2010CRA,
Jia2010SPA,FanLiu2010SPL,FanJiang2012JAMC,FanJiang2013AMSE,Fan2015JMAA,
Fan2016SPA,Fan2016SPL,Fan2018JOTP,XiaoFan2020KM} when the generators $g$ have a linear/sub-linear growth in the unknown variable $z$. The reader is also referred to
\cite{Bahlali2001CRAS,Bahlali2002ECP,BahlaliEssakyHassaniPardoux2002CRAS,
LepeltierSanMartin2002Bernoulli,
BahlaliElAsri2012BSM,BahlaliEssakyHassani2015SIAM,BahlaliKebiri2017Stochastics,
BahlaliEddahbiOuknine2017AoP,Yanghanlin2017Arxiv} when the generators $g$ have a super-linear growth in the unknown variable $z$. The second one is devoted to the bounded solution of BSDEs when the generators $g$ have a quadratic/super-quadratic growth in the unknown variable $z$, see for example \cite{LepeltierSanMartin1998Stochsatics,Kobylanski2000AP,Fan2016SPA,
BriandLepeltierSanMartin2007Bernoulli,Tevzadze2008SPA,DelbaenHuBao2011PTRF,
BriandElie2013SPA,BarrieuElKaroui2013AoP,HuTang2016SPA,FanLuo2017BKMS} for more details. The last one concerns the weakest possible integrability of $(\xi, f)$ for existence and uniqueness of  adapted solution of BSDEs when the generators $g$ have some  growth and/or continuity  in $(y,z)$. Such a study is  dated back to  \cite{BriandHu2006PTRF,BriandHu2008PTRF,DelbaenHuRichou2011AIHPPS, Richou2012SPA,DelbaenHuRichou2015DCDS} for the quadratic BSDEs, and subsequently developed  in \cite{HuTang2018ECP,BuckdahnHuTang2018ECP,FanHu2019ECP,OKimPak2021CRM} for the linearly growing BSDEs, and recently sprang up  in \cite{FanHu2021SPA,FanHuTang2023SPA,FanHuTang2023SCL,FanHuTang2024SCL2}  when the generator $g$ has a sub-quadratic, super-linear or logarithmic sub-linear growth in the unknown variable $z$.  The so-called localization procedure, $\theta$-difference technique and the test function method were combined to  obtain  the following existence and uniqueness of a BSDE when the generator $g$ satisfies~\eqref{eq:1.2}.

Firstly, suppose that the generator $g$ has a unilateral linear growth in $y$ and a linear growth in $z$, i.e., it satisfies \eqref{eq:1.2} with $(\delta, \lambda,  \alpha)=(0, 0, 1)$. It is well known that if the data $|\xi|+\int_0^T f_s {\rm d}s\in L^p$ for $p>1$, then BSDE$(\xi,g)$ admits a solution in $\s^p\times\mcal^p$, and the solution is unique when $g$ is further uniformly Lipschitz continuous  in $(y,z)$. The reader is referred to \cite{PardouxPeng1990SCL,ElKarouiPengQuenez1997MF,LepeltierSanMartin1997SPL,
BriandDelyonHu2003SPA,FanJiang2012JAMC} for more details. Recently,  \cite{HuTang2018ECP,BuckdahnHuTang2018ECP,FanHu2019ECP,OKimPak2021CRM} obtained existence of an unbounded solution to a linearly growing BSDE$(\xi,g)$ under  the more general condition $|\xi|+\int_0^T f_s {\rm d}s\in L\exp(\mu\sqrt{2\ln L})$ for some $\mu\geq \gamma \sqrt{T}$ (which is weaker than $L^p$-integrability ($p>1$) and stronger than $L\ln L$-integrability). They also established uniqueness of the unbounded solution provided that $g$ is  monotone  in $y$ and uniformly Lipschitz continuous  in $z$. Generally speaking, the generator $g$ allows a general growth in $y$ when $g$ is  monotone  in $y$. Relevant works are available in \cite{Pardoux1999Nonlinear,BriandCarmona2000IJSA,BriandDelyonHu2003SPA,
LepeltierMatoussiXu2005AdAP,BriandLepeltierSanMartin2007Bernoulli,FanJiang2013AMSE,
Fan2015JMAA,LionnetReisSzpruch2015AAP,Fan2016SPA,LionnetReisSzpruch2018AAP}.

Secondly, suppose that the generator $g$ has a unilateral linear growth in $y$ and a power sub-linear growth in $z$, i.e., it satisfies  \eqref{eq:1.2} with $(\delta, \lambda) = (0, 0)$ and $\alpha\in (0,1)$. Briand et al.~\cite{BriandDelyonHu2003SPA} show that if the data $|\xi|+\int_0^T f_s {\rm d}s\in L^1$, then BSDE$(\xi,g)$ admits a solution $(Y_t,Z_t)_{t\in\T}$ such that $(Y_t)_{t\in\T}$ is of class (D), and the solution is unique when $g$ further satisfies a monotonicity in $y$ and the uniformly Lipschitz continuity  in $z$. See for example \cite{BriandHu2006PTRF,Fan2016SPA,Fan2018JOTP} for more details. Very recently,  the three authors \cite{FanHuTang2023SCL}  prove an existence and uniqueness result for the $L^1$ solution when the generator $g$ has a unilateral linear growth in $y$ and a logarithmic sub-linear growth in $z$, i.e., it satisfies~\eqref{eq:1.2} with $(\delta, \alpha)= (0, 1)$ and $\lambda\in (-\infty,-1/2)$, see also \cite{FanHuTang2024SCL2} for deeper discussions.

Thirdly, suppose that the generator $g$ has a unilateral linear/super-linear growth in $y$ and a logarithmic super-linear growth in $z$, i.e., it satisfies~\eqref{eq:1.2} with $\delta\in [0,1]$, $\alpha=1$ and $\lambda\in [0,+\infty)$. Let $p:=\delta\vee (\lambda+{1\over 2})\vee (2\lambda)\in [{1\over 2},+\infty)$. Very recently, it was shown in \cite{FanHuTang2023SPA} that if the data $|\xi|+\int_0^T f_s {\rm d}s\in L\exp(\mu (\ln L)^p)$ for some $\mu>\mu_0$ with some value $\mu_0$, then BSDE$(\xi,g)$ admits a solution $(Y_t,Z_t)_{t\in\T}$ such that $(|Y_t|\exp(\mu(t)(\ln(e+|Y_t|))^p))_{t\in\T}$ is of class (D) for some nonnegative and increasing function $\mu(t)$ defined on $\T$ with $\mu(T)=\mu$, and the solution is unique when the generator $g$ further satisfies an extended monotonicity  in $y$ and a uniform continuity  in $z$, or a convexity/concavity  in $(y,z)$, see assumptions (UN1)-(UN3) in \cite{FanHuTang2023SPA} for more details. Furthermore, \cite{BahlaliElAsri2012BSM,BahlaliKebiri2017Stochastics} verified existence of a solution to BSDE$(\xi,g)$ in the space of $\s^p\times\mcal^2$ for some sufficiently large $p>2$, when the data $|\xi|+\int_0^T f_s {\rm d}s\in L^p$ and the generator $g$ satisfies \eqref{eq:1.2} with $(\delta, \alpha)=(1, 1)$, $\lambda={1\over 2}$ and $|g(\omega,t,y,z)|$ being the left side of \eqref{eq:1.2}. They also proved uniqueness of the solution when $g$ further satisfies a  local monotonicity  in $(y,z)$. Related works on super-linearly growing BSDEs are available in \cite{Bahlali2001CRAS,Bahlali2002ECP,BahlaliEssakyHassaniPardoux2002CRAS, LepeltierSanMartin1998Stochsatics,BahlaliEssakyHassani2010CRM,
BahlaliHakassouOuknine2015Stochastics,
BahlaliEssakyHassani2015SIAM,LionnetReisSzpruch2016ArXiv}, where the solution of BSDE$(\xi,g)$ in the space of $\s^p\times\mcal^p$ is considered under  the data $|\xi|+\int_0^T f_s {\rm d}s\in L^p$ for some $p>1$, and several kinds of locally Lipschitz continuity  or local monotonicity  of $g$ in $(y,z)$ are usually used  to guarantee uniqueness of the solution of BSDE$(\xi,g)$.

Fourthly, suppose that the generator $g$ has a unilateral linear growth in $y$ and a sub-quadratic growth in $z$, i.e., it satisfies~\eqref{eq:1.2} with $(\delta, \lambda)=(0, 0)$ and $\alpha\in (1,2)$. Let $\alpha^*$ represent the conjugate of $\alpha$. It was proved in \cite{FanHu2021SPA} that if the data $|\xi|+\int_0^T f_s {\rm d}s\in \exp(\mu L^{2 \over \alpha^*})$ for some $\mu>\mu_0$ with a certain value $\mu_0$, which is weaker than $\exp(\mu L)$-integrability and stronger than $L^p$-integrability ($p>1$), then BSDE$(\xi,g)$ admits a solution $(Y_t,Z_t)_{t\in\T}$ such that $(\exp(\mu(t)|Y_t|^{2\over \alpha^*}))_{t\in\T}$ is of class (D) for some nonnegative and increasing function $\mu(t)$ defined on $\T$ with $\mu(T)=\mu$, and the solution is unique when $|\xi|+\int_0^T f_s {\rm d}s\in \exp(\mu L^{2\over \alpha^*})$ for each $\mu>0$ and the generator $g$ further satisfies an extended convexity/concavity in $(y,z)$, see assumption (H2') in \cite{FanHu2021SPA} for more details.

Finally, suppose that the generator $g$ has a unilateral linear growth in $y$ and a quadratic growth in $z$, i.e., it satisfies~\eqref{eq:1.2} with $(\delta, \lambda, \alpha)=(0, 0, 2)$. It is well known from \cite{Kobylanski2000AP} that if the data  $|\xi|+\int_0^T |f_s| {\rm d}s\in L^\infty$, then BSDE$(\xi,g)$ admits a solution $(Y_t,Z_t)_{t\in\T}$ such that $(Y_t)_{t\in\T}\in \s^\infty$, and the solution is unique if $g$ further satisfies the uniformly Lipschitz continuity  in $y$ and a locally Lipschitz continuity in $z$. The reader is referred to \cite{BriandElie2013SPA,Fan2016SPA,FanLuo2017BKMS,LuoFan2018SD} for more details on the bounded solution of quadratic BSDEs. Subsequently, \cite{BriandHu2006PTRF,BriandHu2008PTRF,DelbaenHuRichou2011AIHPPS,
DelbaenHuRichou2015DCDS}  proved existence and uniqueness of an unbounded solution to quadratic BSDE$(\xi,g)$  under the data  $|\xi|+\int_0^T f_s {\rm d}s\in \exp(\mu L)$ for $\mu:=\gamma e^{\beta T}$, where $g$ is required to be uniformly Lipschitz continuous in $y$ and convex/concave in $z$ for the uniqueness of the solution, see also \cite{BarrieuElKaroui2013AoP,FanHuTang2020CRM} for more details.
A class of quadratic BSDEs subject to   $L^p$-integrable terminal values $(p>1)$ are studied in recent works, see for example \cite{BahlaliEddahbiOuknine2017AoP,Yanghanlin2017Arxiv,
Bahlali2020GJM,BahlaliTangpi2021Arxiv}. In addition,  \cite{DelbaenHuBao2011PTRF} show that super-quadratic BSDEs~\eqref{eq:1.2} (with $(\delta, \lambda) = (0, 0)$ and $\alpha>2$), are not solvable in general and the solution is not unique even if the solution exists. Some relevant solvability results
under the Markovian setting are available in \cite{DelbaenHuBao2011PTRF,
MasieroRichou2013EJP,Richou2012SPA,CheriditoNam2014JFA}.

The  brief review above is depicted in the following Table 1: existing solvability results. It is flavored with the authors' own tastes,  and is also restricted within the scope of their knowledge.  Certainly, it does not exhaust all the developments of BSDEs in the last half a century, which seems to be an impossible task to the authors within such a very limited space, and is also not the objective of the paper.  We would apologize to all those authors of possibly neglected  papers on BSDEs.

\subsection{Organization of the  paper\vspace{0.2cm}}

The  paper will present a comprehensive theory on the well-posedness of 1D nonlinear BSDEs to  cover  most existing results mentioned above. The rest of this paper is organized as follows. In Section 2, we formulate the test function method and prove with  the combined techniques of a priori estimate and localization a general existence result (see \cref{dfn:2.1} and \cref{thm:2.2}),  which yield via right  test functions several existence theorems on the adapted solutions of 1D BSDEs (see \cref{thm:2.3,thm:2.6})  for both cases of logarithmic quasi-linear growth and sub-quadratic/quadratic growth, respectively. In Section 3, we  focus on the comparison theorems of the adapted solutions of 1D BSDEs for both cases of at most linear growth (see \cref{thm:3.3}) and  super-linear at most quadratic growth (see \cref{thm:3.6}),

\begin{center}
\captionof{table}{existing solvability results}\vspace{-0.1cm}
{\small
\begin{tabular}{
|>{\centering\arraybackslash}m{1.9cm}
|>{\centering\arraybackslash}m{2.4cm}
|>{\centering\arraybackslash}m{3cm}
|>{\centering\arraybackslash}m{3.2cm}|>{\centering\arraybackslash}m{1.4cm}
|>{\centering\arraybackslash}m{1.6cm}|
}
\hline

\parbox[c][1.5cm]{1.9cm}{\centering Value of parameters in Ineq.~\eqref{eq:1.2} }& Space of $\xi+\int_0^T f_s {\rm d}s$ &
Space of solution & \makecell{Further condition on $g$ \\ for uniqueness} & References & Related references\\

\hline

\multirow{4}{*}{\parbox{2cm}{\centering$\delta=0$ \\ $\alpha=1$ \\ $\lambda=0$}} & \parbox[c][1.3cm]{2.4cm}{\centering $L^p \text{ for } \ p>1$} & $(Y,Z)\in \s^p\times\mcal^p$ &
uniformly Lipschitz continuous in $(y,z)$ & \cite{PardouxPeng1990SCL,ElKarouiPengQuenez1997MF,LepeltierSanMartin1997SPL,
BriandDelyonHu2003SPA,FanJiang2012JAMC} & \multirow{4}{*}{\parbox{1.8cm}{\centering\cite{Pardoux1999Nonlinear,BriandCarmona2000IJSA,BriandDelyonHu2003SPA,
LepeltierMatoussiXu2005AdAP,BriandLepeltierSanMartin2007Bernoulli,FanJiang2013AMSE,
Fan2015JMAA,LionnetReisSzpruch2015AAP,Fan2016SPA,LionnetReisSzpruch2018AAP}}} \\

\cline{2-5}

& \makecell{$L\exp(\mu\sqrt{2\ln L})$ \\ for some \\ $\mu\geq \gamma \sqrt{T}$} & \parbox[c][2cm]{3cm}{\centering $|Y|\exp(\mu(t)\sqrt{2\ln |Y|})$ is of class (D) with $\mu(T)=\mu$} & monotone in $y$ and uniformly Lipschitz continuous in $z$ & \cite{HuTang2018ECP,BuckdahnHuTang2018ECP,FanHu2019ECP,OKimPak2021CRM} & \\

\hline

\makecell{$\delta=0$\\ $\alpha\in (0,1)$\\ $\lambda=0$} & $L^1$ & $Y$ is of class (D) & monotone in $y$ and uniformly Lipschitz continuous in $z$ & \cite{BriandDelyonHu2003SPA} & \cite{BriandHu2006PTRF,Fan2016SPA,Fan2018JOTP} \\

\hline

\makecell{$\delta=0$\\ $\alpha=1$\\ $\lambda\in (-\infty,-{1\over 2})$} & $L^1$ & $Y$ is of class (D) & extended monotonicity in $y$ and logarithmic uniform continuity in $z$ & \cite{FanHuTang2023SCL} & \cite{FanHuTang2024SCL2} \\

\hline

\makecell{$\delta\in [0,1]$\\ $\alpha=1$\\ $\lambda\in [0,+\infty)$}& $L\exp(\mu (\ln L)^p)$ for some $\mu>\mu_0$, where $p:=\delta\vee (\lambda+{1\over 2})\vee (2\lambda)$ & $|Y|\exp(\mu(t)(\ln(e+|Y|))^p)$ is of class (D) with $\mu(T)=\mu$ & extended monotonicity in $y$ and uniform continuity in $z$, or convexity/ concavity in $(y,z)$ & \cite{FanHuTang2023SPA} & / \\

\hline

\makecell{$\delta=1$\\ $\alpha=1$\\ $\lambda={1\over 2}$}& $L^p$ for a sufficiently large $p>2$ & $(Y,Z)\in \s^p\times \mcal^2$ & some local monotonicity in $(y,z)$& \cite{BahlaliElAsri2012BSM,BahlaliKebiri2017Stochastics} & \cite{Bahlali2001CRAS,Bahlali2002ECP,BahlaliEssakyHassaniPardoux2002CRAS, LepeltierSanMartin1998Stochsatics,BahlaliEssakyHassani2010CRM,
BahlaliHakassouOuknine2015Stochastics,
BahlaliEssakyHassani2015SIAM,LionnetReisSzpruch2016ArXiv}\\

\hline

\makecell{$\delta=0$\\ $\alpha\in (1,2)$\\ $\lambda=0$}& $\exp(\mu L^{2\over \alpha^*})$ for some $\mu>\mu_0$, where $\alpha^*$ is conjugate of $\alpha$& \makecell{$\exp(\mu(t)|Y|^{2\over \alpha^*})$\\ is of class (D)\\ with $\mu(T)=\mu$}& extended convexity/ concavity in $(y,z)$; for each $\mu(t)\equiv \mu>0$& \cite{FanHu2021SPA}& / \\

\hline

\multirow{4}{*}{\parbox{2cm}{\centering $\delta=0$\\ $\alpha=2$\\ $\lambda=0$}} & $L^\infty$& $Y\in \s^\infty$ & uniformly Lipschitz continuous in $y$ and locally Lipschitz continuous in $z$ &\cite{Kobylanski2000AP} & \cite{BriandElie2013SPA,Fan2016SPA,FanLuo2017BKMS,LuoFan2018SD}\\

\cline{2-6}

&\makecell{$\exp(\mu L)$\\ with $\mu:=\gamma e^{\beta T}$}& \makecell{$\exp(\mu(t) |Y|)$ \\ is of class (D)\\ with $\mu(t):=\gamma e^{\beta t}$}& uniformly Lipschitz continuous in $y$ and convex/concave in $z$&\cite{BriandHu2006PTRF,BriandHu2008PTRF,DelbaenHuRichou2011AIHPPS,
DelbaenHuRichou2015DCDS}& \cite{BarrieuElKaroui2013AoP,FanHuTang2020CRM}\\
\hline
\end{tabular}
}\vspace{0.2cm}
\end{center}

\noindent respectively, and establish some existence and uniqueness results (see \cref{thm:3.12,thm:3.13}). We first give a crucial a priori estimate (see \cref{pro:3.1}) associated with the test function, and prove \cref{thm:3.3,thm:3.6} with the proper test function and  the $\theta$-difference technique, respectively. This yields naturally the desired uniqueness results. Some examples and remarks are provided in the last two sections to illustrate the preceding results. See \cref{rmk:2.4,rmk:2.5*,rmk:2.7*,rmk:2.7} and \cref{ex:2.1,ex:2.2} in Section 2 as well as \cref{rmk:3.2,rmk:3.5,rmk:3.6,rmk:3.8*,rmk:3.8,rmk:3.13*,rmk:3.14*}, \cref{ex:3.1,ex:3.2,ex:3.3} and \cref{pro:3.4} in Section 3. In Section 4, several practical applications of our results are introduced, including the conditional $g$-expectation (see \cref{def:4.1} and \cref{pro:4.2,pro:4.3}), the dynamic utility process (see \cref{pro:4.4} and \cref{thm:4.5}), risk measure (\cref{ex:4.8}) and nonlinear Feynman-Kac formula (see \cref{def:4.7} and \cref{thm:4.8}), and some commentaries on known related works are also made, see \cref{rmk:4.3*,rmk:4.6,rmk:4.11}. Finally in Section 5,  we list several open problems on 1D BSDEs, and in Appendix,  we prove a key inequality (see \cref{pro:2.5}) used in Section 2, which is interesting in its own right.

\subsection{Notations and spaces\vspace{0.2cm}}

In this subsection, we give some necessary notations and spaces used in this paper. Let $\R_+:=[0,+\infty)$. For $a,b\in \R$, we denote $a\wedge b:=\min\{a,b\}$, $a^+:=\max\{a,0\}$ and $a^-:=-\min\{a,0\}$, and ${\rm sgn}(x):={\bf 1}_{x>0}-{\bf 1}_{x\leq 0}$, where ${\bf 1}_A$ is the indicator function of set $A$. Denote by ${\bf S}$  the set of $\R_+$-valued continuously differentiable functions $\phi(s,x)$ defined on $[0,T]\times\R_+$ such that $\phi_s(\cdot,\cdot)\geq 0$, $\phi_x(s,\cdot)>0$ and $\phi_{xx}(s,\cdot)>0$, where $\phi_s(\cdot,\cdot)$ is the first-order partial derivative of $\phi(\cdot,\cdot)$ with respect to the first argument, and by $\phi_x(\cdot,\cdot)$ and $\phi_{xx}(\cdot,\cdot)$ respectively the first- and second-order partial derivative of $\phi(\cdot,\cdot)$ with respect to the second argument. Denote by ${\bf \bar S}$ the set of $\R_+$-valued functions $h(t,x,\bar x)$ defined on $\T\times\R_+\times \R_+$ such that $h(t,\cdot,\bar x)$ is nondecreasing for each $(t,\bar x)\in \T\times \R_+$. Denote by $\s^\infty(\T;\R)$ (or $\s^\infty$)  the set of $(\F_t)$-adapted and continuous bounded real processes $(Y_t)_{t\in\T}$. For each $p>0$, let $\s^p(\T;\R)$ (or $\s^p$) be the set of $(\F_t)$-adapted and continuous real processes $(Y_t)_{t\in\T}$ satisfying
$$\|Y\|_{{\s}^p}:=\E\left[\sup_{t\in\T} |Y_t|^p\right]^{{1\over p}\wedge 1}<+\infty,\vspace{0.2cm}$$
and $\mcal^p(\T;\R^d)$ (or $\mcal^p$) the set of all $(\F_t)$-adapted $\R^d$-valued processes $(Z_t)_{t\in\T}$ satisfying
$$
\|Z\|_{\mcal^p}:=\E\left[\left(\int_0^T |Z_t|^2{\rm d}t\right)^{p/2}\right]^{{1\over p}\wedge 1}<+\infty.\vspace{0.1cm}
$$
Denote by $\Sigma_T$ the set of all $(\F_t)$-stopping times $\tau$ valued in $\T$. An $(\F_t)$-adapted real process $(X_t)_{t\in\T}$ is said to be of class (D), if the family $\{X_\tau: \tau\in \Sigma_T\}$ is uniformly integrable.

Now, fix $t\in \T$. For $p,\mu>0$, we denote by $L^p(\F_t)$ and $L^\infty(\F_t)$ the set of $\F_t$-measurable real random variables $\xi$ such that $\E[|\xi|^p]<+\infty$ and $|\xi|\leq M$ for some real $M>0$, respectively, and define the following three spaces of $\F_t$-measurable real random variables:
$$
L(\ln L)^p (\F_t):=\left\{\xi\in \F_t\left| \E\left[|\xi|(\ln(e+|\xi|))^p\right]<+\infty\right.\right\},
$$
$$
L\exp[\mu(\ln L)^p](\F_t):=\left\{\xi\in \F_t\left| \E\left[|\xi|\exp{\left(\mu (\ln(e+|\xi|))^p\right)}\right]<+\infty\right.\right\}
$$
and
$$
\exp(\mu L^p)(\F_t):=\left\{\xi\in \F_t\left| \E\left[\exp{\left(\mu |\xi|^p\right)}\right]<+\infty\right.\right\}.
$$
It is clear that for each $0<p<q$ and $0<\bar \mu,\tilde \mu<\mu$, we have
$$
L^\infty(\F_t)\subset L^q(\F_t)\subset L^p(\F_t), \ \ \ L(\ln L)^q(\F_t)\subset L(\ln L)^p(\F_t),
$$
$$
L\exp[\mu(\ln L)^q](\F_t)\subset L\exp[\bar\mu(\ln L)^q](\F_t)\subset L\exp[\tilde\mu(\ln L)^p](\F_t),
$$
and
$$
\exp(\mu L^q)(\F_t)\subset \exp(\bar\mu L^q)(\F_t)\subset \exp(\tilde\mu L^p)(\F_t).
$$
It can also be verified that for each $\mu,\bar\mu,r>0$ and $0<p<1<q$, we have
$$
 L^\infty(\F_t)\subset \exp(\mu L^r)(\F_t)\subset L\exp[\bar\mu(\ln L)^q](\F_t) \subset L\exp[\mu \ln L](\F_t)=L^{1+\mu}(\F_t)
$$
and
$$
L^q(\F_t)\subset L\exp[\mu(\ln L)^p](\F_t)\subset L(\ln L)^r(\F_t)\subset L^1(\F_t).
$$
Furthermore, for each $p,\mu>0$ and $0<\bar p\leq 1<\tilde p$, the spaces
$$
L(\ln L)^p(\F_t),\ \ L\exp[\mu(\ln L)^{\bar p}](\F_t),\ \ \bigcup\limits_{\bar\mu>\mu} L\exp[\bar\mu(\ln L)^{\tilde p}](\F_t)\ \  {\rm and}\ \ \bigcap\limits_{\bar\mu>0}\exp(\bar\mu L^p)(\F_t)
$$
are all the Orlicz hearts,  corresponding to the following Young functions
$$
x(\ln(e+x))^p,\ \ x\exp[\mu(\ln(e+x))^{\bar p}],\ \ x\exp[\mu(\ln(e+x))^{\tilde p}]\ \ {\rm and}\ \ \exp(x^p)-1,
$$
respectively. More details on the Orlicz space, the Orlicz class and the Orlicz heart are refereed to \cite{EdgarSucheston1992book,CheriditoLi2009MF}. Finally, in all notations of the spaces of random variables, the $\sigma$-algebra $(\F_T)$ is usually omitted whenever  there is no confusion.

\section{Existence results}
\label{sec:2-Existence}
\setcounter{equation}{0}

\subsection{The test function method and a general existence result\vspace{0.2cm}}

Let us first introduce the following assumptions on the generator $g$.
\begin{enumerate}

\item [(EX1)] $\as$, $g(\omega,t,\cdot,\cdot)$ is continuous.

\item [(EX2)] There exist two $\R_+$-valued functions $H(\omega,t,x)$ and $\Gamma(\omega,t,x)$ defined on $\Omega\times\T\times\R_+$, which are $(\F_t)$-progressively measurable for each $x\in\R_+$ and nondecreasing with respect to the variable $x$, such that $\as$, for each $(y,z)\in \R\times\R^d$,
    $$
    |g(\omega,t,y,z)|\leq H(\omega,t,|y|)+\Gamma (\omega,t,|y|) |z|^2,
    $$
    and $\ps$, for each $x\in \R_+$, $\Gamma(\omega,\cdot,x)$ is left-continuous on $\T$ and
    $$
    \int_0^T H(\omega,t,x){\rm d}t+\sup\limits_{t\in\T}\Gamma(\omega,t,x)<+\infty.
    $$

\item [(EX3)] There exists a function $h(\cdot,\cdot,\cdot)\in {\bf\bar S}$ such that $\as$, for each $(y,z)\in \R\times\R^d$,
    $$
    {\rm sgn}(y)g(\omega,t,y,z)\leq f_t(\omega)+h(t,|y|,|z|).
    $$
\end{enumerate}

\begin{defn}\label{dfn:2.1}
Assume that the generator $g$ satisfies assumption (EX3). A function $\varphi(\cdot,\cdot)\in {\bf S}$ is called a test function for $g$ or $h$ if it
satisfies that for each $(s,x,\bar x)\in [0,T]\times \R_+\times\R_+$,
\begin{equation}\label{eq:2.1}
-\varphi_x(s,x)h(s,x,\bar x)+{1\over 2}\varphi_{xx}(s,x)|\bar x|^2+\varphi_s(s,x)\geq 0.
\end{equation}
\end{defn}

\begin{thm}\label{thm:2.2}
Assume that $\xi$ is an $\F_T$-measurable random variable and the generator $g$ satisfies the above assumptions (EX1)-(EX3). If there exists a test function $\varphi(\cdot,\cdot)\in {\bf S}$ for $g$ such that
\begin{equation}\label{eq:2.2}
\E\left[\varphi\left(T,\ |\xi|+\int_0^T f_s{\rm d}s\right)\right]<+\infty,
\end{equation}
then BSDE$(\xi,g)$ admits a solution $(Y_t,Z_t)_{t\in\T}$ such that the process
$$\varphi(t, |Y_t|+\int_0^t f_s{\rm d}s), \quad t\in \T$$
 is of class (D). Furthermore, the process $(\varphi(t, |Y_t|+\int_0^t f_s{\rm d}s))_{t\in \T}$ is a sub-martingale. In particular, we have\vspace{0.1cm}
\begin{equation}\label{eq:2.3}
\varphi\left(t, |Y_t|+\int_0^t f_s{\rm d}s\right)\leq \E\left[\left.\varphi\left(T, |\xi|+\int_0^T f_s{\rm d}s\right)\right|\F_t\right],\ \ t\in\T.\vspace{0.2cm}
\end{equation}
\end{thm}

\begin{proof} The proof consists of the following two steps.\vspace{0.2cm}

{\bf Step 1.} We first prove the inequality~\eqref{eq:2.3}  for  a solution
$$(Y_t,Z_t)_{t\in\T}\in \s^\infty(\T;\R)\times \mcal^2(\T;\R^d)$$
  of BSDE$(\xi,g)$ if further   $|\xi|+\int_0^T f_s{\rm d}s\in L^\infty$. Define
$$
\bar Y_t:=|Y_t|+\int_0^t f_s {\rm d}s\ \ \ \
{\rm and}\ \ \ \ \bar Z_t:={\rm sgn}(Y_t)Z_t,\ \ \ \ t\in \T.
\vspace{0.1cm}
$$
We have  from It\^{o}-Tanaka's formula that\vspace{0.1cm}
$$
\bar Y_t=\bar Y_T+\int_t^T \left({\rm sgn}(Y_s)g(s,Y_s,Z_s)-f_s\right){\rm d}s-\int_t^T \bar Z_s \cdot {\rm d}B_s-\int_t^T {\rm d}L_s, \ \ \ t\in\T,
$$
where $L_\cdot$ is the local time of $Y$ at $0$. Applying It\^{o}'s formula to the process $\varphi(s, \bar Y_s)$ and using assumption (EX3), we have that for each $s\in\T$,
$$
\begin{array}{lll}
\Dis {\rm d}\varphi(s,\bar Y_s)
&=&\Dis \varphi_x(s,\bar Y_s)
\left(-{\rm sgn}(Y_s)g(s,Y_s,Z_s)+f_s\right){\rm d}s+\varphi_x(s,\bar Y_s)\bar Z_s\cdot  {\rm d}B_s\vspace{0.1cm}\\
&&\Dis +\varphi_x(s,\bar Y_s){\rm d}L_s+{1\over 2}\varphi_{xx}(s,\bar Y_s)|\bar Z_s|^2{\rm d}s+\varphi_s(s,\bar Y_s){\rm d}s\vspace{0.2cm}\\
&\geq &\Dis \left[-\varphi_x(s,\bar Y_s)h(s,|Y_s|,|Z_s|)+{1\over 2}\varphi_{xx}(s,\bar Y_s)|Z_s|^2+\varphi_s(s,\bar Y_s)\right]{\rm d}s\\[3mm]
&&+\varphi_x(s,\bar Y_s)\bar Z_s\cdot  {\rm d}B_s.
\end{array}
$$
Since $|Y_s|\leq \bar Y_s$ and  $h(t,\cdot,\bar x)$ is nondecreasing, we see from~\eqref{eq:2.1}  that
$$
{\rm d}\varphi(s,\bar Y_s)\geq \varphi_x(s,\bar Y_s)\bar Z_s\cdot  {\rm d}B_s,\ \ s\in \T,
$$
which yields that
$$
\varphi(T,\bar Y_T)-\varphi(t,\bar Y_t)\geq \int_t^T \varphi_x(s,\bar Y_s)\bar Z_s\cdot  {\rm d}B_s,\ \ t\in \T.
$$
Since $|\xi|+\int_0^T f_s{\rm d}s\in L^\infty$ and $(Y_t,Z_t)_{t\in\T}\in \s^\infty(\T;\R)\times \mcal^2(\T;\R^d)$,  taking the  expectation conditioned on $\F_t$ on both sides of the last inequality,  we have~\eqref{eq:2.3}.\vspace{0.2cm}

{\bf Step 2.} Based on Step 1, we use the localization procedure of Briand and Hu~\cite{BriandHu2006PTRF} to construct the desired solution. For each $n,p\geq 1$ and $(\omega,t,y,z)\in \Omega\times\T\times\R\times\R^d$, denote
\begin{equation}\label{eq:2.4}
\xi^{n,p}:=\xi^+\wedge n-\xi^-\wedge p\ ~{\rm and}\ ~ g^{n,p}(\omega,t,y,z):=g^+(\omega,t,y,z)\wedge n-g^-(\omega,t,y,z)\wedge p.
\end{equation}
It is clear that $|\xi^{n,p}|\leq |\xi|\wedge (n\vee p)$ and $|g^{n,p}|\leq |g|\wedge (n\vee p)$ for each $(y,z)\in \R\times \R^d$. It can also be verified that the generator $g^{n,p}$ satisfies assumptions (EX1)-(EX3) with $f_\cdot\wedge (n\vee p)$ instead of $f_\cdot$ Then, according to \cite{Kobylanski2000AP}, the following BSDE$(\xi^{n,p},g^{n,p})$ has a maximal bounded solution $(Y^{n,p}_t,Z^{n,p}_t)_{t\in\T}$ in  $\s^\infty(\T;\R)\times\mcal^2(\T;\R^d)$:
\begin{equation}\label{eq:2.5}
  Y^{n,p}_t=\xi^{n,p}+\int_t^T g^{n,p}(s,Y^{n,p}_s,Z^{n,p}_s){\rm d}s-\int_t^T Z^{n,p}_s \cdot {\rm d}B_s, \ \ t\in\T.
\end{equation}
The comparison theorem shows that $(Y^{n,p}_t)_{t\in \T}$ is nondecreasing in $n$ and non-increasing in $p$. Furthermore,  we know from Step 1 that for each $n,p\geq 1$,
\begin{equation}\label{eq:2.6}
\begin{array}{l}
\Dis \varphi\left(t, |Y_t^{n,p}|+\int_0^t \left[f_s\wedge (n\vee p)\right]{\rm d}s\right) \\[3mm]
\ \ \leq  \Dis \E\left[\left.\varphi\left(T, |\xi^{n,p}|+\int_0^T \left[f_s\wedge (n\vee p)\right]{\rm d}s\right)\right|\F_t\right]\vspace{0.2cm}\\
\ \ \leq \Dis \E\left[\left.\varphi\left(T, |\xi|+\int_0^T f_s{\rm d}s\right)\right|\F_t\right],\ \ t\in\T.\vspace{0.1cm}
\end{array}
\end{equation}
Now, for each pair of integers $m,l\geq 1$, we define the following stopping times:
\[
 \tau_m:= T\wedge\inf\left\{t\in \T: \E\left[\left.\varphi\left(T, |\xi|+\int_0^T f_s{\rm d}s\right)\right|\F_t\right] \geq \varphi(t,m)\right\}
\]
and
\[
\sigma_{m,l}:=\tau_m\wedge\inf\bigg\{t\in \T: \int_0^t H(s, m){\rm d}s+\sup\limits_{s\in [0,t]}\Gamma(s,m) \geq l\bigg\}, \vspace{0.2cm}
\]
with the convention that $\inf\emptyset=+\infty$. Then the pair
$$(Y^{n,p}_{m,l}(t), Z^{n,p}_{m,l}(t))_{t\in \T}:=(Y^{n,p}_{t\wedge\sigma_{m,l}}, Z^{n,p}_t{\bf 1}_{t\leq\sigma_{m,l}})_{t\in \T}$$
 is a solution in the space of processes $\s^\infty(\T;\R)\times\mcal^2(\T;\R^d)$ to the following BSDE:
\[
Y^{n,p}_{m,l}(t)=Y^{n,p}_{\sigma_{m,l}}+\int_t^T{{\bf 1}_{s\leq \sigma_{m,l}}g^{n,p}(s, Y^{n,p}_{m,l}(s), Z^{n,p}_{m,l}(s)){\rm d}s}-\int_t^T Z^{n,p}_{m,l}(s) \cdot {\rm d}B_s,\ \ t\in\T.
\]
Observe that for each fixed $m,l\geq 1$, $(Y^{n,p}_{m,l}(t))_{t\in \T}$ is nondecreasing in $n$ and non-increasing in $p$, and that $\as$, $(g^{n,p})_{n,p}$ converges locally uniformly in $(y,z)$ to $g$ as $n,p\To \infty$. Since $\varphi(t,\cdot)$ is nondecreasing for each $t\in\T$, by \eqref{eq:2.6} together with the definitions of $\tau_m$ and $\sigma_{m,l}$, we have for $\as$,
$$
\sup\limits_{n,p\geq 1}|Y^{n,p}_{m,l}(t)|\leq m.
$$
Furthermore, since $|g^{n,p}|\leq |g|$ and $g$ satisfies assumption (EX2), we know that $\ass$,
\[
 \RE (y, z)\in [-m, m]\times\R^d,\ \ \
 \sup\limits_{n,p\geq 1}\left({\bf 1}_{s\leq \sigma_{m,l}}\left|g^{n,p}(s, y, z)\right|\right)\leq {\bf 1}_{s\leq \sigma_{m,l}}H(s, m)+ l|z|^2
\]
with $\int_0^T {\bf 1}_{s\leq \sigma_{m,l}} H(s, m){\rm d}s\leq l$. Thus,  for each fixed $m,l\geq 1$, we can apply the stability result for the bounded solutions of BSDEs (see for example Proposition 3.1 in \cite{LuoFan2018SD}). Setting $Y_{m,l}(t):=\inf_p\sup_n Y^{n,p}_{t\wedge\sigma_{m,l}}$, then $(Y_{m,l}(t))_{t\in\T}$ is continuous and the sequence $(Z^{n,p}_{t}{\bf 1}_{t\le \sigma_{m,l}})_{t\in [0,T]}$ converges to  $(Z_{m,l}(t))_{t\in [0,T]}$   strongly  in $\mcal^2(\T;\R^d)$ as $n,p\To\infty$ such that for $t\in \T$,
$$
Y_{m,l}(t)=\inf_{p\geq 1}\sup_{n\geq 1} Y^{n,p}_{\sigma_{m,l}}+\int_t^T{{\bf 1}_{s\le \sigma_{m,l}}g(s,Y_{m,l}(s),Z_{m,l}(s)){\rm d}s}-\int_t^T{Z_{m,l}(s)\cdot {\rm d}B_s}.
$$
Finally, in view of the last equation and  the stability of stopping times $\tau_m$ and $\sigma_{m,l}$, since  we have $\as$, for each $m,l\geq 1$,
\[
Y_{m+1,l+1}(t\wedge \sigma_{m,l})=Y_{m,l+1}(t\wedge \sigma_{m,l})=Y_{m,l}(t\wedge \sigma_{m,l})=\inf_{p\geq1}\sup_{n\geq 1} Y^{n,p}_{t\wedge\sigma_{m,l}}
\]
and\vspace{-0.1cm}
\[
Z_{m+1,l+1}{\bf 1}_{t\leq \sigma_{m,l}}=Z_{m,l+1}{\bf 1}_{t\leq \sigma_{m,l}}=Z_{m,l}{\bf 1}_{t\leq \sigma_{m,l}}=\lim_{n,p\To\infty}Z^{n,p}_{t}{\bf 1}_{t\le \sigma_{m,l}},\vspace{0.1cm}
\]
we see that $(Y_t,Z_t)_{t\in\T}$ is an adapted solution of BSDE$(\xi,g)$, where
$$
Y_t:=\inf_p\sup_n Y^{n,p}_{t} \ \ {\rm and}\ \ Z_t:=\sum_{m=1}^{+\infty}\left( \sum_{l=1}^{+\infty} Z_{m,l}(t){\bf 1}_{t\in [\sigma_{m,l-1},\sigma_{m,l})}\right){\bf 1}_{t\in [\tau_{m-1},\tau_m)}$$
for $ t\in\T,$
with $\tau_0:=0$ and $\sigma_{m,0}:=0$ for each $m\geq 1$. And, \eqref{eq:2.3} follows from \eqref{eq:2.6} by sending $n,p\to \infty$.
Moreover, according to \eqref{eq:2.5}, in a way similar to Step 1,  we also verify that for each $n,p\geq 1$ and $0\leq t\leq r\leq T$,
$$
\varphi\left(t, |Y_t^{n,p}|+\int_0^t \left[f_s\wedge (n\vee p)\right]{\rm d}s\right) \leq \E\left[\left.\varphi\left(r, |Y_r^{n,p}|+\int_0^r \left[f_s\wedge (n\vee p)\right]{\rm d}s\right)\right|\F_t\right].
$$
Thus, in view of \eqref{eq:2.6}, setting $n, p\to \infty$ and using Lebesgue's dominated convergence theorem in the last inequality,  we  see that the process $(\varphi(t, |Y_t|+\int_0^t f_s{\rm d}s))_{t\in \T}$ is indeed a sub-martingale. The proof is then complete.
\end{proof}

As applications of \cref{thm:2.2}, we shall prove Theorems \ref{thm:2.3} and \ref{thm:2.6} below, where the function $h(\cdot,\cdot,\cdot)$ in (EX3) takes the following form:
\begin{equation}\label{eq:2.7}
h(t,x,\bar x):=\beta x\left(\ln(e+x)\right)^\delta+\gamma \bar x^\alpha \left(\ln(e+\bar x)\right)^\lambda,\ \ (t,x,\bar x)\in \T\times \R_+\times\R_+.
\end{equation}
Both theorems can be compared to existing existence results  (for example, see \cite{Kobylanski2000AP,BriandHu2006PTRF,HuTang2018ECP,FanHu2019ECP,
FanHu2021SPA,FanHuTang2023SCL,FanHuTang2023SPA}) on adapted solutions of one-dimensional BSDEs. \vspace{0.2cm}

\subsection{The logarithmic quasi-linear growth case\vspace{0.2cm}}

Let us first consider the case that the generator $g$ has a logarithmic quasi-linear growth in the unknown variables $(y,z)$, i.e., the case of $\alpha=1$ in \eqref{eq:2.7}.

\begin{thm}\label{thm:2.3}
Assume that $\xi$ is an $\F_T$-measurable random variable and the generator $g$ satisfies (EX1)-(EX3) with $h(\cdot,\cdot,\cdot)$ being defined in \eqref{eq:2.7} for $\alpha=1$. Then, the following assertions hold.

(i) Let $\delta=0$ and $\lambda\in (-\infty,-{1\over 2})$. If $\xi+\int_0^T f_s{\rm d}s\in L^1$, then BSDE$(\xi,g)$ admits a solution $(Y_t,Z_t)_{t\in \T}$ such that the process $(Y_t)_{t\in\T}$ is of class (D);

(ii) Let $\delta=0$ and $\lambda=-{1\over 2}$. If $\xi+\int_0^T f_s{\rm d}s\in L(\ln L)^p$ for some $p>0$, then BSDE$(\xi,g)$ admits a solution $(Y_t,Z_t)_{t\in \T}$ such that the process
$$|Y_t|(\ln (e+|Y_t|))^p, \quad t\in \T$$
is of class (D);

(iii) If $p:=\delta\vee (\lambda+{1\over 2})\vee (2\lambda)\in (0,+\infty)$ and $\xi+\int_0^T f_s{\rm d}s\in \cap_{\mu>0}L\exp[\mu(\ln L)^p]$, then BSDE$(\xi,g)$ has a solution $(Y_t,Z_t)_{t\in \T}$ such that the process
$$|Y_t|\exp \left(\mu(\ln (e+|Y_t|))^p\right), \quad t\in\T$$
 is of class (D) for each $\mu>0$;

(iv) Let $\delta=0$ and $\lambda=0$. If $\xi+\int_0^T f_s{\rm d}s\in L^p$ for some $p>1$, then BSDE$(\xi,g)$ admits a solution $(Y_t,Z_t)_{t\in \T}$ such that the process $(|Y_t|^p)_{t\in\T}$
 is of class (D).\vspace{0.2cm}
\end{thm}

\begin{rmk}\label{rmk:2.4}
When the generator $g$ grows faster in both unknown variables $(y,z)$, a stronger integrability  on $\xi+\int_0^T f_s{\rm d}s$ is required for the existence of solution of BSDE$(\xi,g)$. Under a slightly stronger growth condition than (EX2), Assertions  (i), (iii) with $\lambda>0$, and (iv) of \cref{thm:2.3} were given in \cite{FanHuTang2023SCL,FanHuTang2023SPA,Fan2016SPL}, respectively. However, for $\lambda\in [-\frac{1}{2},0)$, Assertions (ii) and (iii) of \cref{thm:2.3} seem to be new.
\end{rmk}

\begin{ex}\label{ex:2.1}
Consider the following simple BSDE:
\begin{equation}\label{eq:2.7*}
Y_t=\xi+\int_t^T (f_s+\beta |Y_s|+\gamma |Z_s|){\rm d}s-\int_t^T Z_s\cdot {\rm d}B_s,\ \ t\in\T,
\end{equation}
where $\xi\geq 0$ and $f_\cdot\in L^1(0,T)$. It is the special case of $\delta=0$, $\lambda=0$ and $\alpha=1$ in \eqref{eq:2.7}. Theorem 2.1 of \cite{HuTang2018ECP} states that BSDE \eqref{eq:2.7*} admits a solution $(Y_t,Z_t)_{t\in \T}$ such that $Y\geq 0$ if and only if there exists a locally bounded process $\bar Y$ such that
$$
\esup\limits_{q\in \acal\ } \left\{\E_q\left[\left.e^{\beta(T-t)}|\xi|\right|\F_t\right]\right\}
+\int_t^T e^{\beta(s-t)} f_s {\rm d}s\leq \bar Y_t,
$$
where $\acal$ is the set of $(\F_t)$-progressively measurable $\R^d$-valued process $(q_t)_{t\in \T}$ such that $|q|\leq \gamma$, and
$$
\frac{{\rm d}\Q^q}{{\rm d}\p}:=M^q_T
$$
with
$$
M^q_t:=\exp\left\{\int_0^t q_s\cdot  {\rm d}B_s-{1\over 2}\int_0^t |q_s|^2{\rm d}s  \right\},\ \ t\in \T\vspace{0.2cm}
$$
and $\E_q$ is the expectation operator with respect to $\Q^q$. In particular, if BSDE \eqref{eq:2.7*} admits a solution $(Y_t,Z_t)_{t\in \T}$ such that $Y\geq 0$, then $\xi e^{\gamma |B_T|}\in L^1$.

The following example is taken from Example 2.3 of \cite{HuTang2018ECP}. Set $d=1$, $T=1$, $\gamma=1$ and
the terminal variable
$$\xi:=e^{{1\over 2}(|B_1|-1)^2}-1.$$
Then, BSDE \eqref{eq:2.7*} has no solution $(Y_t,Z_t)_{t\in \T}$ such that $Y\geq 0$, as $\xi e^{|B_1|}\notin L^1$:
$$
\E\left[(\xi+1) e^{|B_1|}\right]=\frac{1}{\sqrt{2\pi}}\int_{-\infty}^{+\infty} e^{{1\over 2}(|x|-1)^2} e^{|x|}e^{-{1\over 2}x^2}{\rm d}x=+\infty.
$$
Whereas it can be directly checked that this $\xi$ belongs to the space
$$\cap_{0<\mu<\sqrt{2}} L\exp[\mu(\ln L)^p]$$ and then $\cap_{q>0} L(\ln L)^q$, but does not belong to $L\exp[\sqrt{2}(\ln L)^p]$, where $p={1\over 2}$ is defined in (iii) of \cref{thm:2.3}.

Furthermore, \cite{HuTang2018ECP,FanHu2019ECP} show that for a linearly growing BSDE$(\xi,g)$ to have a solution, an $L\exp(\mu\sqrt{\ln L})$-integrability  of $\xi$ is sufficient for $\mu=\gamma\sqrt{2T}$, but  not for any $\mu\in (0, \gamma\sqrt{2T}\,)$, which can not follow from (iii) of \cref{thm:2.3}.
\end{ex}

The following remark further illustrates that the $L\exp(\gamma\sqrt{2T\ln L})$-integrability condition is very close to the weakest one on the terminal value $\xi$ for the existence of an adapted solution of BSDE \eqref{eq:2.7*}, which seems to be new. For this, for each integer $n\geq 1$, by induction we define the following functions
$$
\ln^{(1)}(x):=\ln x,\ \ x\geq e^{(1)}\ \ {\rm and}\ \  \ln^{(n)}(x):=\ln^{(n-1)}(\ln^{(1)}(x))=\ln^{(1)}(\ln^{(n-1)}(x))$$
for $ x\geq e^{(n)},$
where
$$
e^{(1)}:=e\ \ {\rm and}\ \ e^{(n)}:=e^{e^{(n-1)}}.
$$
And, define the function
$$
\Ln(x):=(e^{(n)}+x)\cdot \prod\limits_{i=1}^{n}\ln^{(i)}( e^{(n)}+x), \quad x\ge 0.
$$

\begin{rmk}\label{rmk:2.5*}
By an example, we verify that for any positive integral $n\geq 1$, the integrability
$$
\frac{L\exp\left(\gamma\sqrt{2T\ln L}\right)}{\ln^{(n)}( e^{(n)}+L)}
$$
on the terminal value $\xi$ is not sufficient for the existence of  an adapted solution of BSDE \eqref{eq:2.7*}. More specifically, we set $d=1$, $T=1$, $\gamma=1$ and
$$\xi:=\frac{e^{{1\over 2}(|B_1|-1)^2}}{\Ln(|B_1|)}-1.$$
It is easily checked that $\xi e^{|B_1|}\notin L^1$ due to
$$\begin{array}{lll}
\E\left[(\xi+1) e^{|B_1|}\right]&=&\Dis \frac{1}{\sqrt{2\pi}}\int_{-\infty}^{+\infty} \frac{e^{{1\over 2}(|x|-1)^2}}{\Ln(|x|)} e^{|x|}e^{-{1\over 2}|x|^2}{\rm d}x\\[3mm]
&=&\Dis \left. \frac{2e^{1\over 2}}{\sqrt{2\pi}}\ln\left(\ln^{(n)}(e^{(n)}+|x|)\right)\right|_0^{+\infty}=+\infty.
\end{array}
$$
Consequently, by Theorem 2.1 of \cite{HuTang2018ECP} we conclude that BSDE \eqref{eq:2.7*} has no solution $(Y_t,Z_t)_{t\in \T}$ such that $Y\geq 0$. Whereas, we also directly verify that this $\xi$ belongs to the space $$\frac{L\exp(\sqrt{2\ln L})}{\ln^{(n)}(e^{(n)}+L)},$$
which yields the desired result.

In fact, it is clear that there exists a constant $x_0>0$ such that when $|x|>x_0$,
$$
\frac{e^{{1\over 2}(|x|-1)^2}}{\Ln(|x|)}-1\geq |x|\ \ {\rm and\ then}\ \ \ln^{(n)}\left(e^{(n)}+\frac{e^{{1\over 2}(|x|-1)^2}}{\Ln(|x|)}-1\right)\geq \ln^{(n)}\left(e^{(n)}+|x|\right).
$$
Furthermore, by a simple computation we get
$$
\sqrt{2\ln (\xi+1)}=\sqrt{(|B_1|-1)^2-2\ln(\Ln(|B_1|))}.
$$
Finally, observing that
$$\begin{array}{l}
\lim\limits_{x\To \infty}
\left(\sqrt{(|x|-1)^2-2\ln(\Ln(|x|))}-(|x|-1)\right)\\[3mm]
\ \ =\Dis -\lim\limits_{x\To \infty} \frac{2\ln(\Ln(|x|))}{\sqrt{(|x|-1)^2-2\ln(\Ln(|x|))}+(|x|-1)}=0
\end{array}
$$
and
$$\begin{array}{lll}
\Dis \int_{|x|>x_0} \frac{1}
{\Ln(|x|)\ln^{(n)}\left(e^{(n)}+|x|\right)}
{\rm d}x&=& \Dis \left. \frac{-2}{\ln^{(n)}\left(e^{(n)}+|x|\right)}\right|_{x_0}^{+\infty}\\[3mm]
&=&\Dis \frac{2}{\ln^{(n)}\left(e^{(n)}+|x_0|\right)}<+\infty,
\end{array}
$$
we obtain that there exists a constant $C>0$ depending only on $x_0$ such that
$$
\begin{array}{l}
\Dis \E\left[\frac{(\xi+1)\exp\left(\sqrt{2\ln (\xi+1)}\right)}{\ln^{(n)}\left(e^{(n)}+\xi\right)}\right]\leq \Dis \E\left[\frac{(\xi+1)\exp\left(\sqrt{2\ln (\xi+1)}\right)}{\ln^{(n)}\left(e^{(n)}+|B_1|\right)}{\bf 1}_{|B_1|>x_0}\right]+C
\vspace{0.3cm}\\
\ \ =\Dis \frac{1}{\sqrt{2\pi}}\int_{|x|>x_0} \frac{e^{{1\over 2}(|x|-1)^2}}{\Ln(|x|)}\cdot \frac{e^{\sqrt{(|x|-1)^2-2\ln(\Ln(|x|))}}}{\ln^{(n)}\left(e^{(n)}+|x|\right)}
e^{-{1\over 2}|x|^2}{\rm d}x+C\vspace{0.3cm}\\
\ \ =\Dis \frac{1}{\sqrt{2\pi}}\int_{|x|>x_0} \frac{e^{\sqrt{(|x|-1)^2-2\ln(\Ln(|x|))}-(|x|-1)}}
{\Ln(|x|)\ln^{(n)}\left(e^{(n)}+|x|\right)}
e^{-{1\over 2}}{\rm d}x+C<+\infty.\vspace{0.2cm}
\end{array}
$$
\end{rmk}

To prove \cref{thm:2.3}, we introduce the following \cref{pro:2.5}, whose proof is given in Appendix.

\begin{pro}\label{pro:2.5}
For each $p>1$ and $\lambda\in \R$, there exists a sufficiently large positive constant $k_{\lambda,p}\geq e$ depending only on $(\lambda,p)$ such that for each $k\geq k_{\lambda,p}$,
\begin{equation}\label{eq:2.8}
2xy\left(\ln (k+y)\right)^\lambda\leq px^2\left(\ln (k+x)\right)^{2\lambda}+y^2,\ \ \RE x,y>0.
\end{equation}
In particular, when $p=1$, there is no constant $k$ such that \eqref{eq:2.8} holds true unless $\lambda=0$.
\end{pro}

\begin{proof}[Proof of \cref{thm:2.3}]
Since the generator $g$ satisfies (EX1)-(EX3) with $h(\cdot,\cdot,\cdot)$ being defined in \eqref{eq:2.7} for $\alpha=1$, a function $\varphi(\cdot,\cdot)\in {\bf S}$ is a test function for $g$ if for each $(s,x,\bar x)\in [0,T]\times \R_+\times\R_+$, it holds that
$$
-\varphi_x(s,x)\left(\beta x\left(\ln(e+x)\right)^\delta+\gamma \bar x\left(\ln(e+\bar x)\right)^\lambda\right)+{1\over 2}\varphi_{xx}(s,x)|\bar x|^2+\varphi_s(s,x)\geq 0.
$$
Here, we emphasis that without loss of generality, the constant $e$ in the last inequality can be replaced with any given constant $k\geq e$ due to the fact that $\lim\limits_{x\To \infty}\ln(k+x)/\ln(e+x)=1$. It follows from \cref{pro:2.5} with $p=2$ that there exists a sufficiently large positive constant $k_{\lambda}\geq e$ depending only on $\lambda$ such that for each $k\geq k_\lambda$ and $(s,x,\bar x)\in [0,T]\times \R_+\times\R_+$,
$$
\begin{array}{l}
\Dis -\gamma\varphi_x(s,x) \bar x\left(\ln(k+\bar x)\right)^\lambda+{1\over 2}\varphi_{xx}(s,x)|\bar x|^2\vspace{0.2cm}\\
\ \ =\Dis \frac{\varphi_{xx}(s,x)}{2}\left(-2\frac{\gamma \varphi_x(s,x)}{\varphi_{xx}(s,x)}\bar x\left(\ln(k+\bar x)\right)^\lambda+|\bar x|^2\right)\vspace{0.2cm}\\
\ \ \geq \Dis -\frac{\gamma^2 (\varphi_x(s,x))^2}{\varphi_{xx}(s,x)} \left(\ln\left(k+\frac{\gamma \varphi_x(s,x)}{\varphi_{xx}(s,x)} \right)\right)^{2\lambda}.\vspace{0.1cm}
\end{array}
$$
Thus, if a function $\varphi(\cdot,\cdot)\in {\bf S}$ satisfies that for some $k\geq k_\lambda\geq e$ and each $(s,x)\in [0,T]\times \R_+$,
\begin{equation}\label{eq:2.9}
\begin{array}{l}
-\beta \varphi_x(s,x)(k+x)\left(\ln(k+x)\right)^\delta -\frac{\gamma^2 (\varphi_x(s,x))^2}{\varphi_{xx}(s,x)} \left(\ln\left(k+\frac{\gamma \varphi_x(s,x)}{\varphi_{xx}(s,x)} \right)\right)^{2\lambda}\vspace{0.2cm}\\
\ \ \Dis \hspace{3cm}+\varphi_s(s,x)\geq 0,
\end{array}
\end{equation}
then it is a test function for the generator $g$.\vspace{0.2cm}

(i) Let $\delta=0$ and $\lambda\in (-\infty,-{1\over 2})$. By a similar computation as in \cite{FanHuTang2023SCL}, one can verify that for  sufficiently large  $k\geq k_\lambda$  and  $c\geq 2\beta-\frac{8\gamma^2}{1+2\lambda}$, the following function
$$
\varphi(s,x)=(k+x)\left(1-\left(\ln(k+x)\right)^{1+2\lambda}\right)\exp(c s),\ \ (s,x)\in \T\times\R_+
$$
satisfies the inequality \eqref{eq:2.9} with $\delta=0$ and $\lambda<-1/2$, and thus is a test function for the generator $g$. Since
$$
\lim\limits_{x\To +\infty}\frac{\varphi(s,x)}{x}=\exp(cs)\in [1,\exp(cT)], \quad s\in\T,
$$
we see  from \cref{thm:2.2} that  BSDE$(\xi,g)$ admits a solution $(Y_t,Z_t)_{t\in \T}$ such that the process $(Y_t)_{t\in \T}$ is of class (D) if $\xi+\int_0^T f_s{\rm d}s\in L^1$.\vspace{0.2cm}

(ii) Let $\delta=0$ and $\lambda=-{1\over 2}$. It is not very hard to verify that for  $p>0$, sufficiently large $k\geq k_\lambda$  and $c\geq 2\beta+\frac{4\gamma^2}{p}$, the following function
$$
\varphi(s,x)=(k+x)\left(\ln(k+x)\right)^p \exp(c s),\ \ (s,x)\in \T\times\R_+
$$
satisfies the inequality \eqref{eq:2.9} with $\delta=0$ and $\lambda=-1/2$, and thus is a test function for the generator $g$. Since
$$
\lim\limits_{x\To +\infty}\frac{\varphi(s,x)}{x\left(\ln(e+x)\right)^p }=\exp(cs)\in [1,\exp(cT)], \quad s\in\T,
$$
we see  from \cref{thm:2.2} that if $\xi+\int_0^T f_s{\rm d}s\in L(\ln L)^p$ for some $p>0$, then BSDE$(\xi,g)$ admits a solution $(Y_t,Z_t)_{t\in \T}$ such that the process
$$|Y_t|(\ln (e+|Y_t|))^p, \quad t\in \T$$ is of class (D). \vspace{0.2cm}

(iii) Let $p:=\delta\vee (\lambda+{1\over 2})\vee (2\lambda)\in (0,+\infty)$. By a similar computation as in \cite{FanHuTang2023SPA}, one can verify that for  sufficiently large $k\geq k_\lambda$, $c_1\geq 1$ and $c_2\geq (p+1)\beta-4^{\lambda^+}\gamma^2$, the following function
$$
\varphi(s,x)=(k+x)\exp\left(c_1\exp(c_2 s) \left(\ln(k+x)\right)^p\right),\ \ (s,x)\in \T\times\R_+
$$
satisfies the inequality \eqref{eq:2.9} with $p>0$, and thus is a test function for the generator $g$. Since
$$
\lim\limits_{x\To +\infty}\frac{\varphi(s,x)}{x\exp\left(c_1\exp(c_2 s) \left(\ln(e+x)\right)^p\right)}=1, \quad s\in\T,
$$
we see  from \cref{thm:2.2} that if $\xi+\int_0^T f_s{\rm d}s\in \cap_{\mu>0} L\exp[\mu(\ln L)^p]$, then BSDE$(\xi,g)$ admits a solution $(Y_t,Z_t)_{t\in \T}$ such that for each $\mu>0$, the process
$$|Y_t|\exp \left(\mu(\ln (e+|Y_t|))^p\right), \quad t\in\T$$
 is of class (D).\vspace{0.2cm}

(iv) Let $\delta=0$ and $\lambda=0$. It is easy to verify that for each $p>1$, $k\geq k_\lambda$ and $c\geq p\beta+\frac{p}{p-1}\gamma^2$, the following function
$$
\varphi(s,x)=(k+x)^p \exp(c s),\ \ (s,x)\in \T\times\R_+
$$
satisfies the inequality \eqref{eq:2.9} with $\delta=0$ and $\lambda=0$, and thus is a test function for the generator $g$. Since
$$
\lim\limits_{x\To +\infty}\frac{\varphi(s,x)}{x^p }=\exp(cs)\in [1,\exp(cT)], \quad s\in\T,
$$
we see  from \cref{thm:2.2} that if $\xi+\int_0^T f_s{\rm d}s\in L^p$ for some $p>1$, then BSDE$(\xi,g)$ admits a solution $(Y_t,Z_t)_{t\in \T}$ such that the process $(|Y_t|^p)_{t\in\T}$ is of class (D).\vspace{0.2cm}
\end{proof}

\subsection{The sub-quadratic/quadratic growth case\vspace{0.2cm}}

Let us further consider the case that the generator $g$ has a unilateral linear growth in the unknown variable $y$ and has a super-linear at most quadratic growth in the unknown variable $z$.

\begin{thm}\label{thm:2.6}
Assume that $\xi$ is an $\F_T$-measurable random variable and the generator $g$ satisfies assumptions (EX1)-(EX3) with $h(\cdot,\cdot,\cdot)$ being defined in \eqref{eq:2.7} for $\delta=0$, $\lambda=0$ and $\alpha\in (1,2]$. Let $\alpha^*$ be the conjugate of $\alpha$, i.e.,
$$
\frac{1}{\alpha}+\frac{1}{\alpha^*}=1.\vspace{0.1cm}
$$
If $\xi+\int_0^T f_s{\rm d}s\in \cap_{\mu>0}\exp(\mu L^{2\over \alpha^*})$, then BSDE$(\xi,g)$ admits a solution $(Y_t,Z_t)_{t\in \T}$ such that the process $(\exp(\mu |Y_t|^{2\over \alpha^*}))_{t\in \T}$ is of class (D) for each $\mu>0$. In particular, if $\xi+\int_0^T f_s{\rm d}s\in L^{\infty}$, then BSDE$(\xi,g)$ admits a solution $(Y_t,Z_t)_{t\in \T}$ such that $(Y_t)_{t\in\T}\in \s^{\infty}(\T;\R)$.
\end{thm}

\begin{rmk}\label{rmk:2.7*}
Under a slightly stronger growth condition than (EX2), the assertions stated in \cref{thm:2.6} were given in \cite{FanHu2021SPA,BriandHu2006PTRF,BriandHu2008PTRF,
FanHu2019ECP,Kobylanski2000AP}, respectively.\vspace{0.1cm}
\end{rmk}

\begin{ex}\label{ex:2.2}
Let us consider the following typical BSDE:
\begin{equation}\label{eq:2.10*}
Y_t=\xi+\int_t^T {1\over 2}|Z_s|^2 {\rm d}s-\int_t^T Z_s\cdot {\rm d}B_s,\ \ t\in\T.
\end{equation}
The change of variables leads to the equation
$$
e^{Y_t}=e^\xi-\int_t^T e^{Y_s}Z_s\cdot {\rm d}B_s, \ \ t\in \T,
$$
which has a solution as $e^\xi\in L^1$. On the other hand, since $\{e^{Y_t}\}_{t\in \T}$ is a positive super-martingale, Theorem 3.1 of \cite{BriandLepeltierSanMartin2007Bernoulli} observes that the inclusion $e^\xi\in L^1$ is also  necessary   for this BSDE to have a solution.

Furthermore, we consider the case where the generator $g$ satisfies assumptions (EX1)-(EX3) with $h(\cdot,\cdot,\cdot)$ being defined in \eqref{eq:2.7} for $\delta=0$, $\lambda=0$ and $\alpha=2$. \cite{BriandHu2006PTRF} show that for BSDE$(\xi,g)$ to have a solution, an $\exp(\mu L)$-integrability of $\xi+\int_0^T f_s{\rm d}s$  is sufficient for $\mu=2\gamma e^{\beta T}$, but not for any $\mu \in (0, 2\gamma e^{\beta T})$, which can not follow from \cref{thm:2.6}.
\end{ex}

\begin{proof}[Proof of \cref{thm:2.6}] Consider both cases of $\alpha=2$ and $\alpha\in (1,2)$.\vspace{0.1cm}

(i) The case of $\alpha=2$. In this case, a function $\varphi(\cdot,\cdot)\in {\bf S}$ is a test function for $g$ if
$$
-\varphi_x(s,x)\left(\beta x+ \gamma \bar x^2\right)+{1\over 2}\varphi_{xx}(s,x)|\bar x|^2+\varphi_s(s,x)\geq 0, \quad \forall (s,x,\bar x)\in [0,T]\times \R_+\times\R_+.
$$
It is easy to verify that for each $c_1\geq 2\gamma$ and $c_2\geq \beta$, the following function
$$
\varphi(s,x)=\exp\left(c_1\exp(c_2 s)x \right),\ \ (s,x)\in \T\times\R_+
$$
satisfies the last inequality, and thus is a test function for the generator $g$. It follows from \cref{thm:2.2} that if $\xi+\int_0^T f_s{\rm d}s\in \cap_{\mu>0}\exp(\mu L)$, then BSDE$(\xi,g)$ admits a solution $(Y_t,Z_t)_{t\in \T}$ such that $(\exp\left(\mu |Y_t|\right))_{t\in \T}$ is of class (D) for each $\mu>0$, which is the desired assertion since $\alpha^*=2$ in this case. \vspace{0.2cm}

(ii) The case of $\alpha\in (1,2)$. In this case, a function $\varphi(\cdot,\cdot)\in {\bf S}$ is a test function for $g$ if
$$
-\varphi_x(s,x)\left(\beta x+\gamma\bar x^\alpha\right)+{1\over 2}\varphi_{xx}(s,x)|\bar x|^2+\varphi_s(s,x)\geq 0, \quad \forall (s,x,\bar x)\in [0,T]\times \R_+\times\R_+.
$$
Using Young's inequality, we have that for each $(s,x,\bar x)\in [0,T]\times \R_+\times\R_+$,
$$
\begin{array}{l}
\Dis -\gamma\varphi_x(s,x) \bar x^\alpha+\frac{1}{2}\varphi_{xx}(s,x)|\bar x|^2 \vspace{0.2cm}\\
\ \ =\Dis \varphi_{xx}(s,x)\left(-\frac{\gamma \varphi_x(s,x)}{\varphi_{xx}(s,x)}\bar x^\alpha+\frac{1}{2}|\bar x|^2\right)\vspace{0.2cm}\\
\ \ \geq \Dis -\frac{2-\alpha}{2\alpha}\cdot \frac{(\alpha\gamma\varphi_x(s,x))^{2\over 2-\alpha}}{(\varphi_{xx}(s,x))^{\alpha\over 2-\alpha}}\geq -\frac{(2\gamma\varphi_x(s,x))^{2\over 2-\alpha}}{(\varphi_{xx}(s,x))^{\alpha\over 2-\alpha}}.\vspace{0.1cm}
\end{array}
$$
Thus, if a function $\varphi(\cdot,\cdot)\in {\bf S}$ satisfies
\begin{equation}\label{eq:2.10}
-\beta \varphi_x(s,x)x-\frac{(2\gamma \varphi_x(s,x))^{2\over 2-\alpha}}{(\varphi_{xx}(s,x))^{\alpha\over 2-\alpha}}+\varphi_s(s,x)\geq 0, \quad (s,x)\in [0,T]\times \R_+,
\end{equation}
then it is a test function for the generator $g$. Furthermore, by a similar computation as in \cite{FanHu2021SPA}, it can be verified that for each $c_1\geq 1$, $k\geq k_{\alpha,c_1}$ with $k_{\alpha,c_1}$ being a positive constant depending only on $\alpha$ and $c_1$, and $c_2\geq \beta+(1+c_1)2^{6\over 2-\alpha}(2\alpha-2)^{2-2\alpha\over 2-\alpha}\gamma^{2\over 2-\alpha}$, the following function
$$
\varphi(s,x)=\exp\left(c_1\exp(c_2 s)(x+k)^{2\over \alpha^*} \right),\ \ (s,x)\in \T\times\R_+
$$
satisfies the inequality \eqref{eq:2.10}, and thus is a test function for the generator $g$. Since
$$
\exp\left(x^{2\over \alpha^*} \right)\leq \varphi(s,x)\leq \exp\left(c_1\exp(c_2 T)k^{2\over \alpha^*}\right) \exp\left(c_1\exp(c_2 T)x^{2\over \alpha^*} \right)
$$
for any $(s,x)\in \T\times \R_+$,
we see  from \cref{thm:2.2} that if $\xi+\int_0^T f_s{\rm d}s\in \cap_{\mu>0}\exp(\mu L^{2\over \alpha^*})$, then BSDE$(\xi,g)$ admits a solution $(Y_t,Z_t)_{t\in \T}$ such that the process $(\exp(\mu |Y_t|^{2\over \alpha^*}))_{t\in \T}$ is of class (D) for each $\mu>0$.\vspace{0.2cm}

Finally, by \eqref{eq:2.3} we conclude that if $\xi+\int_0^T f_s{\rm d}s\in L^{\infty}$, then the process  $(\varphi(t,|Y_t|+\int_0^t f_s{\rm d}s))_{t\in \T}$ is bounded, and then $(Y_t)_{t\in \T}\in \s^{\infty}(\T;\R)$. The proof is complete.
\end{proof}

We depict \cref{thm:2.3,thm:2.6} in the following Table 2: existence results  in \cref{thm:2.3,thm:2.6}.

\begin{center}
\captionof{table}{existence results in \cref{thm:2.3,thm:2.6}}\vspace{0.2cm}
{\small
\begin{tabular}
{
|>{\centering\arraybackslash}m{2cm}
|>{\centering\arraybackslash}m{2.5cm}
|>{\centering\arraybackslash}m{3cm}
|>{\centering\arraybackslash}m{3.5cm}
|>{\centering\arraybackslash}m{2.6cm}|}

\hline

\parbox[c][1.2cm]{2cm}{\centering Case}& Subcase& Space of $\xi+\int_0^T f_s {\rm d}s$ &
Space of solution & Existence\\

\hline

\multirow{11}{*}{\parbox{2cm}{\centering Logarithmic quasi-linear growth case: $\alpha=1$}}& \makecell{$\delta=0$\\ $\lambda\in (-\infty,-{1\over 2})$} & $L^1$ & $Y$ is of class (D) &
\makecell{~\\ \cref{thm:2.3} (i)\\~} \\

\cline{2-5}

& \makecell{$\delta=0$\\ $\lambda=-{1\over 2}$} & \makecell{$L(\ln L)^p$\\ for some $p>0$} & \makecell{$|Y|(\ln (e+|Y|))^p$ \\ is of class (D)} &
\makecell{~\\ \cref{thm:2.3} (ii)\\~} \\

\cline{2-5}

& \makecell{$\delta\in [0,1]$\\ $\lambda\in (-{1\over 2},+\infty)$} & $\cap_{\mu>0}L\exp[\mu(\ln L)^p]$ with $p:=\delta\vee (\lambda+{1\over 2})\vee (2\lambda)$ & $\cap_{\mu>0}|Y|\exp[\mu(\ln |Y|)^p]$ is of class (D) & \parbox[c][2cm]{2.8cm}{\centering \cref{thm:2.3} (iii)}\\

\cline{2-5}

& \makecell{$\delta=0$ \\ $\lambda=0$} & $L^p$ for some $p>1$ & $|Y|^p$ is of class (D) & \makecell{~\\  \cref{thm:2.3} (iv)\\~} \\

\hline

\multirow{3}{*}{\parbox{2cm}{\centering Sub-quadratic /quadratic growth case: $\alpha\in (1,2]$}} & \multirow{4}{*}{\parbox{2.5cm}{\centering $\delta=0$\\ $\lambda=0$}} & $\cap_{\mu>0}\exp(\mu L^{2\over \alpha^*})$ with $\alpha^*$ being the conjugate of $\alpha$ & \makecell{$\cap_{\mu>0}\exp(\mu |Y|^{2\over \alpha^*})$\\ is of class (D)} & \multirow{4}{*}{\parbox{2.8cm}{\centering \cref{thm:2.6}}} \\

\cline{3-4}

& & \parbox[c][1.5cm]{3cm}{\centering $L^{\infty}$} & $Y\in \s^{\infty}(\T;\R)$ & \\
\hline
\end{tabular}
}\vspace{0.2cm}
\end{center}

\begin{rmk}\label{rmk:2.7}
Some finer integrability of the data $\xi+\int_0^T f_t{\rm d}t$ might be found  for the existence of an adapted solution to BSDE$(\xi,g)$ with  new test functions different from those in the proof of Theorems~\ref{thm:2.3} and \ref{thm:2.6}. The reader is referred to \cite{BriandHu2006PTRF,FanHu2019ECP,FanHu2021SPA,
FanHuTang2023SPA} for more details.
\end{rmk}

\section{Comparison theorems and existence and uniqueness results}
\label{sec:3-Uniqueness}
\setcounter{equation}{0}

In this section, we prove two comparison theorems under different growth  of the generator $g$ in the unknown variables $(y,z)$, which immediately yield the desired uniqueness  and can be compared to some existing comparison results given in for example \cite{ElKarouiPengQuenez1997MF,Kobylanski2000AP,HuImkeller2005AAP,
BriandHu2008PTRF,Fan2016SPA,FanHu2021SPA,FanHuTang2023SPA,FanHuTang2023SCL}.

First, we have the following a priori estimate.

\begin{pro}\label{pro:3.1}
Assume that there exists a function $h(\cdot,\cdot,\cdot)\in {\bf\bar S}$  such that $\as$,
$$
{\bf 1}_{Y_t>0}\ g(t,Y_t,Z_t)\leq f_t+h(t,Y_t^+,|Z_t|),
$$
and that $\varphi(\cdot,\cdot)\in {\bf S}$ is a test function for $h$. If $(\varphi(t,Y^+_t+\int_0^t f_s{\rm d}s))_{t\in\T}$ is of class (D), then we have
$$
\varphi\left(t, \ Y_t^+ +\int_0^t f_s{\rm d}s\right)\leq \E\left[\left.\varphi\left(T, \ \xi^+ +\int_0^T f_s{\rm d}s\right)\right|\F_t\right],\ \ t\in\T.\vspace{0.1cm}
$$
\end{pro}

\begin{proof}
Note that $(\varphi(t,Y^+_t+\int_0^t f_s{\rm d}s))_{t\in\T}$ belongs to class (D). The desired conclusion can be easily obtained by a similar computation to step 1 in the proof of \cref{thm:2.2} with $Y^+_t$, $Y^+_s$, ${\bf 1}_{Y_t>0}$ and ${\bf 1}_{Y_s>0}$ instead of $|Y_t|$, $|Y_s|$, ${\rm sgn}(Y_t)$ and ${\rm sgn}(Y_s)$, respectively.
\end{proof}

\subsection{Comparison results for the case of an at most linear growth\vspace{0.2cm}}

Now, let us introduce the following assumptions on the generator $g$, where $g$ has a unilateral linear growth in $y$ and an at most  linear growth in $z$.

\begin{enumerate}

\item [(UN1)] $g$ has an extended monotonicity  in $y$, i.e., there exists a continuous, increasing and concave function $\rho(\cdot):\R_+\To \R_+$ satisfying $\rho(0)=0$, $\rho(u)>0$ for $u>0$ and
    $$\int_{0^+}\frac{{\rm d}u}{\rho(u)}:=\lim\limits_{\eps\To 0^+}\int_0^\eps\frac{{\rm d}u}{\rho(u)}=+\infty\vspace{0.1cm}$$
    such that $\as$, for each $(y_1,y_2,z)\in \R\times\R\times\R^d$,
    $$
    {\rm sgn}(y_1-y_2) \left(g(\omega,t,y_1,z)-g(\omega,t,y_2,z)\right)\leq \rho(|y_1-y_2|).\vspace{0.1cm}
    $$

\item [(UN2)] $g$ has a logarithmic uniform continuity  in $z$, i.e., there exist a non-positive constant $\lambda\in (-\infty,0]$ and a nondecreasing continuous function $\kappa(\cdot):\R_+\To \R_+$ with linear growth and $\kappa(0)=0$ such that $\as$, for each $(y,z_1,z_2)\in \R\times \R^d\times \R^d$,
    $$
    \begin{array}{lll}
    |g(\omega,t,y,z_1)-g(\omega,t,y,z_2)|&\leq&\Dis  \kappa\left(|z_1-z_2|(\ln (e+|z_1-z_2|))^\lambda\right)\vspace{0.2cm}\\
    &\leq &\Dis \kappa\left(|z_1-z_2|\right).
    \end{array}
    $$
\end{enumerate}

\begin{rmk}\label{rmk:3.2}
Since the function $\rho(\cdot)$ appearing in (UN1) is nondecreasing and concave with $\rho(0)=0$, we can verify that $\rho(\cdot)$ has a  linear growth. We always assume that there exists a $A>0$ such that
\begin{equation}\label{eq:3.1}
\RE u\in\R_+, \ \ \ \rho(u)\leq A(u+1)\ \ {\rm and}\ \ \kappa(u)\leq A(u+1).
\end{equation}
Thus, if the generator $g$ satisfies (UN1) and (UN2), then we have $\as$, for each $(y,z)\in \R\times \R^d$,
$$
\begin{array}{lll}
{\rm sgn}(y)g(\omega,t,y,z)&\leq & \Dis |g(t,0,0)|+ \rho(|y|)+\kappa(|z|\left(\ln(e+|z|)\right)^\lambda)\\
&\leq & \Dis |g(t,0,0)|+ 2A+A|y|+A|z|\left(\ln(e+|z|)\right)^\lambda,
\end{array}
$$
which means that the generator $g$ has a unilateral linear growth in $y$ and a logarithmic sub-linear/linear growth in the unknown variable $z$, and then satisfies assumption (EX3) with $f_t:=|g(t,0,0)|+2A$ and $h(\cdot,\cdot,\cdot)$ being defined in \eqref{eq:2.7} for $\beta=\gamma=A$, $\delta=0$, $\alpha=1$ and $\lambda\leq 0$. In addition,  in the case of $\lambda=0$, assumption (UN2) is equivalent to saying that the function $g(t,\omega, y,z)$ is uniformly continuous in the variable $z$ uniformly with respect to the variables $(t,\omega,y)$;  the assumption (UN2) becomes stronger  as  $\lambda$ decreases.
\end{rmk}

\begin{thm}\label{thm:3.3}
Let $\xi$ and $\xi'$ be two terminal conditions such that $\xi\leq \xi'$, the generator $g$ (resp. $g'$) satisfy assumptions (UN1) and (UN2), and $(Y_t, Z_t)_{t\in\T}$ and $(Y'_t, Z'_t)_{t\in\T}$ be, respectively, adapted solutions to BSDE$(\xi, g)$ and BSDE$(\xi', g')$ such  that
\begin{equation}\label{eq:3.2}
\begin{array}{c}
{\bf 1}_{Y_t>Y'_t}\left(g(t,Y'_t,Z'_t)-g'(t,Y'_t,Z'_t)\right)\leq 0\vspace{0.2cm}\\
({\rm resp.}\  \ {\bf 1}_{Y_t>Y'_t}\left(g(t,Y_t,Z_t)-g'(t,Y_t,Z_t)\right)\leq 0\ ).
\end{array}
\end{equation}
 Then,  we have $Y_t\leq Y'_t$ for each $t\in\T$,  if either of the following four conditions is true:

(i) $\lambda<-1/2$ and both processes $(Y_t)_{t\in\T}$ and $(Y'_t)_{t\in\T}$ are of class (D).\vspace{0.1cm}

(ii) $\lambda=-1/2$ and both processes $(|Y_t|(\ln(e+|Y_t))^p)_{t\in\T}$ and $(|Y'_t|(\ln(e+|Y'_t))^p)_{t\in\T}$ are of class (D) for some constant $p>0$.\vspace{0.1cm}

(iii) $\lambda\in (-1/2,0]$ and both processes $(|Y_t|\exp(\mu(\ln(e+|Y_t|))^{\lambda+{1\over 2}}))_{t\in\T}$ and $(|Y'_t|\exp(\mu(\ln(e+|Y'_t|))^{\lambda+{1\over 2}}))_{t\in\T}$ are of class (D) for each $\mu>0$. \vspace{0.1cm}

(iv) $\lambda=0$ and both processes $(|Y_t|^p)_{t\in\T}$ and $(|Y'_t|^p)_{t\in\T}$ are of class (D) for some $p>1$.\vspace{0.2cm}
\end{thm}

\begin{rmk}\label{rmk:3.8*}
Assertions (i) and (iv) of \cref{thm:3.3} were given in \cite{FanHuTang2023SCL} and \cite{Fan2016SPA}, respectively. However, Assertions (ii) and (iii) of \cref{thm:3.3} seem to be new.\vspace{0.1cm}
\end{rmk}

El Karoui et al.~\cite{ElKarouiPengQuenez1997MF} show that the strict comparison theorem is true  for solutions of two BSDEs when either of both generators is uniformly Lipschitz continuous in $(y,z)$. However, the following two examples indicate that the strict comparison theorem fails to be true in general when the generator $g$ satisfies only assumptions (UN1) and (UN2), which are provided in Section 5.3 of \cite{PardouxRascanu2014Book} and Example 3.2 of \cite{Jia2010SPA}, respectively. In finance, this means that there are infinitely many opportunities of arbitrage.

\begin{ex}\label{ex:3.1}
Let $d=1$ and consider the following BSDE:
\begin{equation}\label{eq:3.2*}
Y_t=\xi-2\int_t^T \sqrt{Y_s^+}\, {\rm d}s-\int_t^T Z_s\cdot {\rm d}B_s,\ \ t\in\T.
\end{equation}
Clearly, the generator $g(\omega,t,y,z):\equiv -2\sqrt{y^+}$ satisfies assumptions (UN1) and (UN2) with $\rho(x)=x$ and $\kappa(x)\equiv 0$. It is not hard to verify that
$(Y_t,Z_t)_{t\in \T}:=(0,0)_{t\in \T}$ and $(Y'_t,Z'_t)_{t\in \T}:=(t^2,0)_{t\in \T}$ are respectively the unique solution to \eqref{eq:3.2*} with $\xi=0$ and $\xi=T^2$ such that $(|Y_t|^p)_{t\in\T}$ and $(|Y'_t|^p)_{t\in\T}$ belong to class (D) for each $p>0$. Note that $Y'_T=T^2>0=Y_T$, but $Y_0=Y'_0=0$.\vspace{0.2cm}
\end{ex}

\begin{ex}\label{ex:3.2}
Let $d=1$ and consider the following BSDE:
\begin{equation}\label{eq:3.3*}
Y_t=\xi-3\int_t^T|Z_s|^{2\over 3}\,  {\rm d}s-\int_t^T Z_s\cdot {\rm d}B_s,\ \ t\in\T.
\end{equation}
Clearly, the generator $g(\omega,t,y,z):\equiv -3|z|^{2\over 3}$ satisfies (UN1) and (UN2) with $\rho(x)=x$ and $\kappa(x)= 3x^{2\over 3}$. It is not hard to verify that
$(Y_t,Z_t)_{t\in \T}:=(0,0)_{t\in \T}$ and $(Y'_t,Z'_t)_{t\in \T}:=({1\over 4}B_t^4,B_t^3)_{t\in \T}$ are respectively the unique solution to \eqref{eq:3.3*} with $\xi=0$ and $\xi={1\over 4}B_T^4$ such that $(|Y_t|^p)_{t\in\T}$ and $(|Y'_t|^p)_{t\in\T}$ belong to class (D) for each $p>0$. Note that $\p(Y'_T>Y_T)=1>0$, but $Y_0=Y'_0=0$.\vspace{0.2cm}
\end{ex}

\begin{proof}[Proof of \cref{thm:3.3}]
We only prove the case that the generator $g$ satisfies assumptions (UN1) and (UN2), and $\as$,
\begin{equation}\label{eq:3.3}
{\bf 1}_{Y_t>Y'_t}\left(g(t,Y'_t,Z'_t)-g'(t,Y'_t,Z'_t)\right)\leq 0.
\end{equation}
The other case can be proved in the same way. According to Theorem 2.1 in \cite{Fan2016SPA} and the above assumptions, it suffices to prove that the process $((Y_t-Y'_t)^+)_{t\in\T}$ is bounded whenever either of four conditions (i)-(iv) holds. \vspace{0.2cm}

Define $\hat Y:=Y-Y'$ and $\hat Z:=Z-Z'$.
Then, the pair of processes $(\hat Y_t,\hat Z_t)_{t\in\T}$ verifies
\begin{equation}\label{eq:3.4}
  \hat Y_t=\hat\xi+\int_t^T \hat g(s,\hat Y_s,\hat Z_s) {\rm d}s-\int_t^T \hat Z_s \cdot {\rm d}B_s, \ \ \ \ t\in\T,
\end{equation}
where $\hat\xi:=\xi-\xi'$ and
$$\hat g(s,\hat Y_s,\hat Z_s):=g(s,Y_s,Z_s)-g'(s,Y'_s,Z'_s).$$
From assumptions (UN1) and (UN2) together with inequalities \eqref{eq:3.3} and \eqref{eq:3.1}, we have
\begin{equation}\label{eq:3.5}
\begin{array}{l}
{\bf 1}_{\hat Y_t>0}\ \hat g(t,\hat Y_t,\hat Z_t)={\bf 1}_{\hat Y_t>0}\left(g(s,Y_s,Z_s)-g'(s,Y'_s,Z'_s)\right)\vspace{0.2cm}\\
\ \ = \Dis {\bf 1}_{\hat Y_t>0}\Bigl[g(s,Y_s,Z_s)-g(s,Y'_s,Z_s)+ g(s,Y'_s,Z_s)\\[2mm]
\quad\quad  -g(s,Y'_s,Z'_s)+g(s,Y'_s,Z'_s)-g'(s,Y'_s,Z'_s)\Bigr]
\vspace{0.2cm}\\
\ \ \leq \rho(\hat Y^+_t)+\kappa\left(|\hat Z_t|(\ln(e+|\hat Z_t|))^\lambda\right)\vspace{0.2cm}\\
\ \ \leq 2A+A\hat Y^+_t +A|\hat Z_t|\left(\ln(e+|\hat Z_t|)\right)^\lambda=f_t+h(t,\hat Y^+_t,|\hat Z_t|),
\end{array}
\end{equation}
where $f_t:\equiv 2A$ and for each $(t,x,\bar x)\in \T\times\R_+\times\R_+$,
\begin{equation}\label{eq:3.6}
h(t,x,\bar x):=Ax+A\bar x\left(\ln(e+\bar x)\right)^\lambda.
\end{equation}

Using~\cref{pro:3.1} together with \eqref{eq:3.5} and \eqref{eq:3.6},  we now verify that $\hat Y_\cdot^+$ is a bounded process whenever either of conditions (i)-(iv) is true. \vspace{0.2cm}

(i) Let $\lambda<-1/2$ and both processes $(Y_t)_{t\in\T}$ and $(Y'_t)_{t\in\T}$ be of class (D). For each $k\geq e$ sufficient large and each $c\geq 2A-\frac{8A^2}{1+2\lambda}$, define the following function
$$
\varphi(s,x)=(k+x)\left(1-\left(\ln(k+x)\right)^{1+2\lambda}\right)\exp(c s),\ \ (s,x)\in \T\times\R_+.
$$
Since $0\leq \varphi(s,x)\leq (k+x)\exp(c T)$ for each $(s,x)\in \T\times\R_+$,  the process $\{\varphi(t,\hat Y_t^+ +2At)\}_{t\in\T}$ is of class (D). On the other hand, a similar analysis to that in (i) of the proof of \cref{thm:2.3} yields that the last function $\varphi(\cdot,\cdot)$ satisfies \eqref{eq:2.1}, and thus is a test function for $h$ defined in \eqref{eq:3.6}. It then follows from \cref{pro:3.1} that
$$
\varphi\left(t, \ Y_t^+ +2At\right)\leq \E\left[\left.\varphi\left(T, \ \xi^+ +2AT\right)\right|\F_t\right]=\varphi\left(T, 2AT\right),\ \ t\in\T,
$$
which means that $(\hat Y_t^+)_{t\in\T}$ is a bounded process.\vspace{0.2cm}

(ii) Let $\lambda=-1/2$ and both processes $(|Y_t|(\ln(e+|Y_t|))^p)_{t\in\T}$ and $(|Y'_t|(\ln(e+|Y'_t|))^p)_{t\in\T}$ be of class (D) for some constant $p>0$. For each $k\geq e$ sufficient large and each $c\geq 2A+\frac{4A^2}{p}$, define the following function
$$
\varphi(s,x)=(k+x)\left(\ln(k+x)\right)^p \exp(c s),\ \ (s,x)\in \T\times\R_+.
$$
Since $0\leq \varphi(s,x)\leq Kx\left(\ln(e+x)\right)^p$ for each $(s,x)\in \T\times\R_+$ and some positive constant $K>0$ depending only on $(k,T)$, we can deduce that $\{\varphi(t,\hat Y_t^+ +2At)\}_{t\in\T}$ is of class (D). On the other hand, a similar analysis to that in (ii) of the proof of \cref{thm:2.3} yields that the last function $\varphi(\cdot,\cdot)$ satisfies \eqref{eq:2.1}, and hence is a test function for $h$ defined in \eqref{eq:3.6}. Thus, the boundedness of the process $(\hat Y_t^+)_{t\in\T}$ follows immediately as in (i).\vspace{0.2cm}

(iii) Let $\lambda\in (-{1\over 2},0]$ and both processes $(|Y_t|\exp(\mu(\ln(e+|Y_t|))^{\lambda+{1\over 2}}))_{t\in\T}$ and $(|Y'_t|\exp(\mu(\ln(e+|Y'_t|))^{\lambda+{1\over 2}}))_{t\in\T}$ be of class (D) for any $\mu>0$. For each $k\geq e$ sufficient large, $c_1\geq 1$ and $c_2\geq (\lambda+3/2)\beta-4^{\lambda^+}\gamma^2$, define the following function
$$
\varphi(s,x)=(k+x)\exp\left(c_1\exp(c_2 s) \left(\ln(k+x)\right)^{\lambda+{1\over 2}}\right),\ \ (s,x)\in \T\times\R_+.
$$
Since $0\leq \varphi(s,x)\leq K x\exp\left(K\left(\ln(e+x)\right)^{\lambda+{1\over 2}}\right)$ for each $(s,x)\in \T\times\R_+$ and some positive constant $K>0$ depending only on $(k,T)$, we can deduce that $\{\varphi(t,\hat Y_t^+ +2At)\}_{t\in\T}$ is of class (D). On the other hand, a similar analysis to that in (iii) of the proof of \cref{thm:2.3} yields that the last function $\varphi(\cdot,\cdot)$ satisfies \eqref{eq:2.1}, and is a test function for $h$ defined in \eqref{eq:3.6}. Thus, the boundedness of the process $(\hat Y_t^+)_{t\in\T}$ follows immediately as in (i).\vspace{0.2cm}

(iv) Let $\lambda=0$ and both processes $(|Y_t|^p)_{t\in\T}$ and $(|Y'_t|^p)_{t\in\T}$ be of class (D) for some $p>1$. Note that for each $\mu>0$, there exists a positive constant $K>0$ depending only on $(\mu,p)$ such that
$$
0\leq x\exp\left(\mu (\ln(e+x))^{1\over 2}\right)\leq K x^p, \ \ x\geq 1.
$$
The desired assertion is a direct consequence of (iii).\vspace{0.2cm}
\end{proof}

The following example is taken from Remark 6 of \cite{Jia2008CRA}, which indicates that the uniform continuity of the generator $g$ in the unknown variable $y$ is not sufficient for the uniqueness of the solution to a BSDE$(\xi,g)$.

\begin{ex}\label{ex:3.3}
Let us consider the following BSDE:
\begin{equation}\label{eq:3.8*}
Y_t=\int_t^T \sqrt{|Y_s|} {\rm d}s-\int_t^T Z_s\cdot {\rm d}B_s,\ \ t\in\T.
\end{equation}
Clearly, $g(\omega,t,y,z):\equiv \sqrt{|y|}$ is uniformly continuous. It is not hard to check that for each $c\in \T$, the pair of processes
$$
(Y_t,Z_t):={1\over 4}\left([(c-t)^+]^2,0\right), \quad  t\in\T
$$
is a solution to \eqref{eq:3.8*} such that $(|Y_t|^p)_{t\in\T}$ belongs to class (D) for each $p>0$.\vspace{0.2cm}
\end{ex}

\subsection{Comparison results for the super-linear at most quadratic growth case\vspace{0.2cm}}

In the following comparison theorem, we will use the following assumption on the generator $g$, where the generator $g$ admits a super-linear at most quadratic growth in $(y,z)$ in general.

\begin{enumerate}
\item [(UN3)] It holds that $\as$,
   \begin{equation}\label{eq:3.7}
   \begin{array}{l}
     \Dis {\bf 1}_{\delta_\theta y>0} \frac{g(\omega,t,y_1,z)-\theta g(\omega,t,y_2,z)}{1-\theta}\vspace{0.1cm}\\
     \ \ \ \ \Dis \leq f_t(\omega)+\beta(|y_1|+|y_2|)+h(t,(\delta_\theta y)^+, |\delta_\theta z|),\vspace{0.2cm}\\
    \RE (y_1,y_2,z_1,z_2)\in \R\times\R\times\R^d\times\R^d\ \ {\rm and}\ \  \theta\in (0,1),
    \end{array}
   \end{equation}
    or
   \begin{equation}\label{eq:3.8}
   \begin{array}{l}
     \Dis -{\bf 1}_{\delta_\theta y<0} \frac{g(\omega,t,y_1,z)-\theta g(\omega,t,y_2,z)}{1-\theta}\vspace{0.1cm}\\
     \ \ \ \ \Dis \leq f_t(\omega)+\beta(|y_1|+|y_2|)+h(t,(\delta_\theta y)^-, |\delta_\theta z|),\vspace{0.2cm}\\
     \RE (y_1,y_2,z_1,z_2)\in \R\times\R\times\R^d\times\R^d\ \ {\rm and}\ \  \theta\in (0,1),\vspace{0.2cm}
   \end{array}
   \end{equation}
    where $h(\cdot,\cdot,\cdot)\in {\bf\bar S}$ is defined in \eqref{eq:2.7} for $\lambda\geq 0$,
    $$
    \delta_\theta y:=\frac{y_1-\theta y_2}{1-\theta}\ \ {\rm and}\ \ \delta_\theta z:=\frac{z_1-\theta z_2}{1-\theta}.
    $$
\end{enumerate}


\begin{pro}\label{pro:3.4}
Assume that the generator $g$ satisfies (EX3) with $h(\cdot,\cdot,\cdot)\in {\bf\bar S}$ being defined in \eqref{eq:2.7} for $\lambda\geq 0$. Then, assumption (UN3) holds true for $g$ if  either of the following three conditions is true:
\begin{enumerate}
\item [(i)] $\as$, $g(\omega,t,\cdot,\cdot)$ is convex or concave;

\item [(ii)]$\as$, for each $(y,z)\in \R\times\R^d$, $g(\omega,t,\cdot,z)$ is Lipschitz continuous and $g(\omega,t, y,\cdot)$ is convex or concave;

\item [(iii)] $g(t,y,z)\equiv l(y)q(z)$, where both $l:\R\To\R$ and $q:\R^d\To\R$ are bounded and Lipschitz continuous, and the function $q(\cdot)$ has a bounded support.
\end{enumerate}
\end{pro}

The proof is similar to that of Proposition 3.5 of \cite{FanHu2021SPA}, and is omitted here.

\begin{rmk}\label{rmk:3.5}
One typical example of (UN3) is
$g(\omega,t,y,z):=g_1(y)+g_2(y)$,
where $g_1:\R\To\R$ is convex or concave with a unilateral logarithmic sup-linear growth, i.e., there exists a nonnegative constant $A\geq 0$ such that for each $y\in\R$,
$$
{\rm sgn}(y)g_1(y)\leq A+\beta |y|(\ln(e+|y|))^\delta,
$$
and $g_2:\R\To\R$ is a Lipschitz continuous function. In other words, $g$ is a Lipschitz continuous perturbation of some convex (concave) function. Another typical example of (UN3) is
$\bar g(\omega,t,y,z):=g_3(z)+g_4(z)$,
where $g_3:\R^d\To\R$ is convex or concave with a logarithmic sub-quadratic growth, i.e., there exists a nonnegative constant $A\geq 0$ such that for each $z\in\R^d$,
$$
|g_3(z)|\leq A+\gamma |z|^\alpha (\ln(e+|z|))^\lambda,
$$
and $g_4:\R^d\To\R$ is a Lipschitz continuous function with a bounded support. In other words, $\bar g$ is a local Lipschitz continuous  perturbation of some convex (concave) function. Furthermore, it is easy to verify that if both generators $g_1$ and $g_2$ satisfies \eqref{eq:3.7} (resp. \eqref{eq:3.8}), then $g_1+g_2$ also satisfies \eqref{eq:3.7} (resp. \eqref{eq:3.8}). Consequently, the generator $g$ satisfying (UN3) may be not necessarily convex (concave) or Lipschitz continuous  in $(y,z)$, and it can have a general growth in $y$.
\end{rmk}

\begin{thm}\label{thm:3.6}
Suppose that $\xi$ and $\xi'$ are two terminal conditions such that $\xi\leq \xi'$, the generater $g$ (resp. $g'$) satisfies assumption (UN3) with $h(\cdot,\cdot,\cdot)$ being defined in \eqref{eq:2.7} for $\lambda\geq 0$, and $(Y_t, Z_t)_{t\in\T}$ and $(Y'_t, Z'_t)_{t\in\T}$ are, respectively, adapted solutions to BSDE$(\xi, g)$ and BSDE$(\xi', g')$ such that
$$
g(t,Y'_t,Z'_t)\leq g'(t,Y'_t,Z'_t) \ \ ({\rm resp.}\  \ g(t,Y_t,Z_t)\leq g'(t,Y_t,Z_t)\ ).
$$
Then the following two assertions hold true:

(i) Let $\alpha=1$ and $p:=\delta\vee (\lambda+{1\over 2})\vee (2\lambda)$. If $\int_0^T f_s{\rm d}s\in \cap_{\mu>0} L\exp[\mu(\ln L)^p]$ and both processes $(|Y_t|\exp(\mu(\ln(e+|Y_t|))^p))_{t\in\T}$ and $(|Y'_t|\exp(\mu(\ln(e+|Y'_t|))^p))_{t\in\T}$ are of class (D) for each $\mu>0$, then for each $t\in\T$, we have $Y_t\leq Y'_t$.

(ii) Let $\delta=0$, $\lambda=0$, $\alpha\in (1,2]$ and $\alpha^*$ be the conjugate of $\alpha$. If $\int_0^T f_s{\rm d}s\in \exp(\mu L^{2\over \alpha^*})$ and both processes
$(\exp(\mu(|Y_t|)^{2\over \alpha^*}))_{t\in\T}$ and $(\exp(\mu(|Y'_t|)^{2\over \alpha^*}))_{t\in\T}$
are of class (D) for each $\mu>0$, then for each $t\in\T$, we have $Y_t\leq Y'_t$. In particular, if the random variable $\int_0^T f_s{\rm d}s$ and both processes $Y$ and $Y'$ are all bounded, then for each $t\in\T$, $Y_t\leq Y'_t$. \vspace{0.2cm}
\end{thm}

\begin{rmk}\label{rmk:3.6}
Assertions in \cref{thm:3.6} were given in \cite{FanHuTang2023SPA,FanHu2021SPA,FanHuTang2020CRM}, respectively. Furthermore, let us suppose that the generator $g$ satisfies assumption (EX3) with $h(\cdot,\cdot,\cdot)$ being defined in \eqref{eq:2.7} for $\delta=0$, $\lambda=0$, $\alpha=2$. For the case of the bounded terminal condition, it has been shown in \cite{Fan2016SPA} and \cite{Kobylanski2000AP} that in order to ensure that the (strictly) comparison result in \cref{thm:3.6} holds, the assumption (UN3) can be weakened such that the generator $g$ further satisfies assumption (UN1) and the following locally Lipschitz condition in $z$: $\as$, for each $(y,z_1,z_2)\in \R\times\R^{1\times d}\times\R^{1\times d}$,
\begin{equation}\label{eq:3.12*}
|g(\omega,t,y,z_1)-g(\omega,t,y,z_2)|\leq \gamma (1+|z_1|^\delta+|z_2|^\delta)|z_1-z_2|
\end{equation}
with $\delta\in [0,1]$ and $\gamma>0$. For the case of the unbounded terminal condition, it has been shown in \cite{FanHuTang2020CRM} that in order to ensure that the comparison result in \cref{thm:3.6} holds, the assumption (UN3) can be weakened such that the generator $g$ is uniformly Lipschitz continuous in $y$ and further satisfies \eqref{eq:3.12*} with $\delta\in [0,1)$ and an additional strictly positive/negative quadratic condition of the generator $g$ in $z$, see assumptions (H3) and (H3') in \cite{FanHuTang2020CRM} for more details.\vspace{0.2cm}
\end{rmk}

\begin{proof}[Proof of \cref{thm:3.6}]
(i) We first consider the case when the generator $g$ satisfies \eqref{eq:3.7} in assumption (UN3), and $\as$,
\begin{equation}\label{eq:3.9}
g(t,Y'_t,Z'_t)\leq g'(t,Y'_t,Z'_t).
\end{equation}
The $\theta$-difference technique put forward initially in Briand and Hu~\cite{BriandHu2008PTRF} will be used in the following argument. For each fixed $\theta\in (0,1)$, define
\begin{equation}\label{eq:3.10}
\delta_\theta  U:=\frac{Y-\theta Y'}{1-\theta}\ \  {\rm and} \ \ \delta_\theta  V:=\frac{Z-\theta Z'}{1-\theta}.
\end{equation}
Then the pair $(\delta_\theta  U_t,\delta_\theta  V_t)_{t\in\T}$ verifies the following BSDE:
\begin{equation}\label{eq:3.11}
  \delta_\theta  U_t=\delta_\theta  U_T +\int_t^T \delta_\theta g (s,\delta_\theta  U_s,\delta_\theta  V_s) {\rm d}s-\int_t^T \delta_\theta  V_s \cdot {\rm d}B_s, \ \ \ \ t\in\T,
\end{equation}
where $\ass$,
\begin{equation}\label{eq:3.12}
\begin{array}{l}
\delta_\theta g (s,\delta_\theta  U_s,\delta_\theta  V_s)\vspace{0.1cm}\\
\ \ \ \ \Dis := \frac{(g(s, Y_s, Z_s)-\theta g(s, Y'_s, Z'_s))+\theta (g(s,Y'_s, Z'_s)-g'(s,Y'_s, Z'_s))}{1-\theta}.
\end{array}
\end{equation}
It follows from the assumptions that $\ass$, for each $(y,z)\in \R\times \R^d$,
\begin{equation}\label{eq:3.13}
{\bf 1}_{\delta_\theta U_s>0} \ \delta_\theta g (s,\delta_\theta  U_s,\delta_\theta  V_s)\leq \bar f_s+h(s, (\delta_\theta U_s)^+, |\delta_\theta V_s|)
\end{equation}
with
$$\bar f_s:=f_s+\beta (|Y_s|+|Y'_s|).$$
On the other hand, for each $k\geq e$ sufficient large, $c_1\geq 1$ and $c_2\geq (p+1)\beta-4^{\lambda^+}\gamma^2$, define the following function
$$
\varphi(s,x)=(k+x)\exp\left(c_1\exp(c_2 s) \left(\ln(k+x)\right)^p\right),\ \ (s,x)\in \T\times\R_+.
$$
Since $0\leq \varphi(s,x)\leq K x\exp\left(K\left(\ln(e+x)\right)^p\right)$ for each $(s,x)\in \T\times\R_+$ and some positive constant $K>0$ depending only on $(k,T)$, according to the assumptions it is not hard to verify that the process $\{\varphi(t, (\delta_\theta U_t)^+ +\int_0^t \bar f_s {\rm d}s)\}_{t\in\T}$ is of class (D). On the other hand, by a similar analysis to that in (iii) of the proof of \cref{thm:2.3} we can conclude that the last function $\varphi(\cdot,\cdot)$ satisfies \eqref{eq:2.1}, and hence is a test function for $h$ defined in \eqref{eq:2.7} with $\alpha=1$ and $\lambda\geq 0$. It then follows from \cref{pro:3.1} that
\begin{equation}\label{eq:3.14}
\begin{array}{lll}
(\delta_\theta U_t)^+ &\leq &\Dis \varphi\left(t, (\delta_\theta U_t)^+ +\int_0^t \bar f_s {\rm d}s\right)\\
&\leq& \Dis \E\left[\left.\varphi\left(T, (\delta_\theta U_T)^+ +\int_0^T \bar f_s {\rm d}s \right)\right|\F_t\right],\ \ t\in\T.
\end{array}
\end{equation}
Moreover, since
\begin{equation}\label{eq:3.15}
\delta_\theta  U_T^+=\frac{(\xi-\theta \xi')^+}{1-\theta}=\frac{\left[\xi-\theta \xi+\theta(\xi-\xi')\right]^+}{1-\theta}\leq \xi^+,
\end{equation}
it follows that
$$
(Y_t-\theta Y'_t)^+\leq (1-\theta)\E\left[\left.\varphi\left(T, \xi^+ +\int_0^T \bar f_s {\rm d}s \right)\right|\F_t\right],\ \ t\in\T.\vspace{0.2cm}
$$
Thus, the desired assertion follows by sending $\theta$ to $1$ in the last inequality.\vspace{0.2cm}

For the case that \eqref{eq:3.9} holds and the generator $g$ satisfies \eqref{eq:3.8}, we need to use $\theta Y-Y'$ and $\theta Z-Z'$, respectively, instead of $Y-\theta Y'$ and $Z-\theta Z'$ in \eqref{eq:3.10}. In this case, the generator $\delta_\theta  g$ in \eqref{eq:3.11} and \eqref{eq:3.12} should be
$$
\delta_\theta g (s,\delta_\theta  U_s,\delta_\theta  V_s):= \frac{(\theta g(s, Y_s, Z_s)-g(s,Y'_s, Z'_s))+(g(s,Y'_s, Z'_s)-g'(s,Y'_s, Z'_s))}{1-\theta}.
$$
It follows from \eqref{eq:3.8} that the generator $\delta_\theta g$ still satisfies \eqref{eq:3.13}. Consequently, \eqref{eq:3.14} still holds.
Moreover, by using
\[
\delta_\theta U_T^+=\frac{(\theta \xi-\xi')^+}{1-\theta}=\frac{\left[\theta \xi- \xi+(\xi-\xi')\right]^+}{1-\theta}\leq (-\xi)^+=\xi^-
\]
instead of \eqref{eq:3.15}, by virtue of \eqref{eq:3.14} we deduce that
$$
(\theta Y_t-Y'_t)^+\leq (1-\theta)\E\left[\left.\varphi\left(T, \xi^- +\int_0^T \bar f_s {\rm d}s \right)\right|\F_t\right],\ \ t\in\T.
$$
Thus, the desired assertion follows by sending $\theta$ to $1$ in the last inequality.\vspace{0.2cm}

Finally, in the same way we can prove the desired assertion under the conditions that the generator $g'$ satisfies assumption (UN3) and $\as$, $g(t,Y_t,Z_t)\leq g'(t,Y_t,Z_t)$.\vspace{0.2cm}

(ii) The desired assertion can be proved in the same way as in (i). The only difference lies in that the test function used in (i) needs to be replaced with those used, respectively, in (i) and (ii) of the proof of \cref{thm:2.6} for two different cases of $\alpha=2$ and $\alpha\in (1,2)$.
\end{proof}

\begin{rmk}\label{rmk:3.8}
Some key ideas in the proof of Theorems~\ref{thm:3.3} and \ref{thm:3.6} go back to \cite{BriandHu2006PTRF,Fan2016SPA,FanHu2021SPA,FanHuTang2023SCL,
FanHuTang2024SCL2}.
\end{rmk}

\subsection{Existence and uniqueness\vspace{0.2cm}}

Using Theorems~\ref{thm:2.3}, \ref{thm:2.6}, \ref{thm:3.3} and \ref{thm:3.6}, we easily have the following two existence and uniqueness results, whose proofs are omitted here.

\begin{thm}\label{thm:3.12}
Assume that $\xi$ is an $\F_T$-measurable random variable and the generator $g$ satisfies (EX1)-(EX3) with $h(\cdot,\cdot,\cdot)$ being defined in \eqref{eq:2.7} for $\alpha=1$. Then, the following assertions hold.

(i) Let $\delta=0$ and $\lambda\in (-\infty,-{1\over 2})$. If $\xi+\int_0^T f_s{\rm d}s\in L^1$ and the generator $g$ further satisfies (UN1) and (UN2), then BSDE$(\xi,g)$ admits a unique solution $(Y_t,Z_t)_{t\in \T}$ such that $(Y_t)_{t\in\T}$ is of class (D);

(ii) Let $\delta=0$ and $\lambda=-{1\over 2}$. If $\xi+\int_0^T f_s{\rm d}s\in L(\ln L)^p$ for some $p>0$ and the generator $g$ further satisfies (UN1) and (UN2), then BSDE$(\xi,g)$ admits a unique solution $(Y_t,Z_t)_{t\in \T}$ such that the process $(|Y_t|(\ln (e+|Y_t|))^p)_{t\in\T}$ is of class (D);

(iii) Let $p:=\delta\vee (\lambda+{1\over 2})\vee (2\lambda)\in (0,+\infty)$. If $\xi+\int_0^T f_s{\rm d}s\in \cap_{\mu>0} L\exp[\mu(\ln L)^p]$ and the generator $g$ further satisfies (UN1) and (UN2) for the case of $\lambda\in (-{1\over 2}, 0]$ and (UN3) for the case of $\lambda\in [0,+\infty)$, then BSDE$(\xi,g)$ admits a unique solution $(Y_t,Z_t)_{t\in \T}$ such that the process $(|Y_t|\exp (\mu(\ln (e+|Y_t|))^p))_{t\in\T}$ is of class (D) for each $\mu>0$;

(iv) Let $\delta=0$ and $\lambda=0$. If $\xi+\int_0^T f_s{\rm d}s\in L^p$ for some $p>1$ and $g$ further satisfies (UN1) and (UN2), then BSDE$(\xi,g)$ admits a unique solution $(Y_t,Z_t)_{t\in \T}$ such that $(|Y_t|^p)_{t\in\T}$ is of class (D).\vspace{0.1cm}
\end{thm}

\begin{rmk}\label{rmk:3.13*}
Under a slightly stronger growth condition than (EX2), Assertions (i), (iii) with $\lambda\in [0, +\infty)$, and (iv) of \cref{thm:3.12} were given in \cite{FanHuTang2023SCL,FanHuTang2023SPA,Fan2016SPL,Fan2016SPA}, respectively. However, for $\lambda\in [-\frac{1}{2},0)$, Assertions (ii) and (iii) of \cref{thm:3.12} seem to be new.\vspace{0.1cm}
\end{rmk}

\begin{thm}\label{thm:3.13}
Assume that $\xi$ is an $\F_T$-measurable random variable and the generator $g$ satisfies assumptions (EX1)-(EX3) and (UN3) with $h(\cdot,\cdot,\cdot)$ being defined in \eqref{eq:2.7} for $\delta=0$, $\lambda=0$ and $\alpha\in (1,2]$. If $\xi+\int_0^T f_s{\rm d}s\in \cap_{\mu>0}\exp(\mu L^{2\over \alpha^*})$ with $\alpha^*$ being the conjugate of $\alpha$, then BSDE$(\xi,g)$ admits a unique solution $(Y_t,Z_t)_{t\in \T}$ such that the process $(\exp(\mu |Y_t|^{2\over \alpha^*}))_{t\in\T}$ is of class (D) for each $\mu>0$. In particular, if $\xi+\int_0^T f_s{\rm d}s\in L^{\infty}(\F_T)$, then BSDE$(\xi,g)$ admits a unique solution $(Y_t,Z_t)_{t\in \T}$ such that $(Y_t)_{t\in\T}\in \s^{\infty}(\T;\R)$.\vspace{0.1cm}
\end{thm}

\begin{rmk}\label{rmk:3.14*}
Under a slightly stronger growth condition than (EX2), Assertions in \cref{thm:3.13} were given in \cite{FanHu2021SPA,BriandHu2006PTRF,BriandHu2008PTRF,
FanHu2019ECP,Kobylanski2000AP}, respectively.\vspace{0.1cm}
\end{rmk}

We depict \cref{thm:3.12,thm:3.13} in the following Table 3: existence and uniqueness results in \cref{thm:3.12,thm:3.13}.

\begin{center}
\captionof{table}{existence and uniqueness results in \cref{thm:3.12,thm:3.13}}\vspace{0.2cm}
{\small
\begin{tabular}
{
|>{\centering\arraybackslash}m{2.1cm}
|>{\centering\arraybackslash}m{3cm}
|>{\centering\arraybackslash}m{3.3cm}
|>{\centering\arraybackslash}m{2.4cm}
|>{\centering\arraybackslash}m{2.8cm}|}

\hline

\parbox[c][1.3cm]{2.1cm}{\centering Case}& Space of $\xi+\int_0^T f_s {\rm d}s$ & Space of solution & Condition on $g$ & Existence and uniqueness\\

\hline

\makecell{$\delta=0$\\ $\alpha=1$\\ $\lambda\in (-\infty,-{1\over 2})$} & $L^1$ & $Y$ is of class (D) & (EX1)-(EX3), (UN1)-(UN2)&
\parbox[c][2cm]{2.8cm}{\centering \cref{thm:3.12} (i)} \\

\hline

\makecell{$\delta=0$\\ $\alpha=1$\\ $\lambda=-{1\over 2}$} & \makecell{$L(\ln L)^p$\\ for some $p>0$} & \makecell{$|Y|(\ln (e+|Y|))^p$ \\ is of class (D)} & (EX1)-(EX3), (UN1)-(UN2) &
\parbox[c][2cm]{2.8cm}{\centering \cref{thm:3.12} (ii)} \\

\hline

\makecell{$\delta\in [0,1]$\\ $\alpha=1$\\ $\lambda\in (-{1\over 2},0]$} & $\cap_{\mu>0}L\exp[\mu(\ln L)^p]$ with $p:=\delta\vee (\lambda+{1\over 2})\vee (2\lambda)$ & $\cap_{\mu>0}|Y|\exp[\mu(\ln |Y|)^p]$ is of class (D) & \parbox[c][2cm]{2.4cm}{\centering (EX1)-(EX3), (UN1)-(UN2)} & \multirow{4}{*}{\parbox{2.8cm}{\centering \cref{thm:3.12} (iii)} }\\

\cline{1-4}

\makecell{$\delta\in [0,1]$\\ $\alpha=1$\\ $\lambda\in [0,+\infty)$} & $\cap_{\mu>0}L\exp[\mu(\ln L)^p]$ with $p:=\delta\vee (\lambda+{1\over 2})\vee (2\lambda)$ & $\cap_{\mu>0}|Y|\exp[\mu(\ln |Y|)^p]$ is of class (D) & \parbox[c][2cm]{2.4cm}{\centering (EX1)-(EX3), (UN3)} & \\

\hline

\makecell{ $\delta=0$ \\ $\alpha=1$\\ $\lambda=0$} & $L^p$ for some $p>1$ & $|Y|^p$ is of class (D) & (EX1)-(EX3), (UN1)-(UN2) & \parbox[c][2cm]{2.8cm}{\centering \cref{thm:3.12} (iv)} \\

\hline

\multirow{4}{*}{\parbox{2.5cm}{\centering $\delta=0$\\ $\alpha\in (1,2]$\\ $\lambda=0$}} & $\cap_{\mu>0}\exp(\mu L^{2\over \alpha^*})$ with $\alpha^*$ being the conjugate of $\alpha$ & \makecell{$\cap_{\mu>0}\exp(\mu |Y|^{2\over \alpha^*})$\\ is of class (D)} & \multirow{4}{*}{\parbox{2.4cm} {\centering (EX1)-(EX3),\\ (UN3)}} & \multirow{4}{*}{\parbox{2.8cm}{\centering \cref{thm:3.13}}} \\

\cline{2-3}

& \parbox[c][1.5cm]{3cm}{\centering $L^{\infty}$} & $Y\in \s^{\infty}(\T;\R)$& & \\
\hline
\end{tabular}
}\vspace{0.2cm}
\end{center}

\section{Applications}
\label{sec:4-Application}
\setcounter{equation}{0}

In this section, we will introduce some applications of our theoretical results obtained in the last two sections, which are enlightened by for example \cite{Peng1990SICON,Peng1991Stochastics,Peng1997Book,BriandHu2008PTRF,
Jia2008PHDThesis,FanHu2021SPA,FanHuTang2023SCL,FanHuTang2023SPA}.

\subsection{The (conditional) $g$-expectation defined on $L^1(\F_T)$\vspace{0.2cm}}

First of all, we extend the notion of (conditional) $g$-expectation of \cite{Peng1997Book} defined on the space  $L^2(\F_T)$ of squarely integrable random variables to the larger one $L^1(\F_T)$ of integrable random variables.

\begin{defn}\label{def:4.1}
Let the generator $g$ satisfy assumptions (EX1)-(EX2) and (UN1)-(UN2) with $\lambda\in (-\infty,-{1\over 2})$ and $\int_0^T f_s {\rm d}s\in L^1$. Assume further that $g$ satisfies the following assumption:
\begin{equation}\label{eq:4.1}
\as,\ \ \ g(\omega,t,y,0)\equiv 0,\ \ \RE y\in \R.
\end{equation}
By virtue of (i) in \cref{thm:3.12}, for each $\xi\in L^1(\F_T)$ and $t\in\T$, we can denote the conditional $g$-expectation $\ecal_g[\xi|\F_t]$ of $\xi$ with respect to $\F_t$ by the following formula:
\begin{equation}\label{eq:4.2}
\ecal_g[\xi|\F_t]:=Y_t^\xi,
\end{equation}
where $(Y_t^\xi,Z_t^\xi)_{t\in\T}$ is the unique solution of BSDE$(\xi,g)$ such that $Y_\cdot^\xi$ belongs to class (D). In particular, we call $\ecal_g[\xi]:=\ecal_g[\xi|\F_0]$ the $g$-expectation of $\xi$.\vspace{0.1cm}
\end{defn}

It is clear that the conditional $g$-expectation operator $\ecal_g[\cdot|\F_t]$ defined by \eqref{eq:4.2} maps $L^1(\F_T)$ to $L^1(\F_t)$ for each $t\in\T$, which shares the same domain with the classical mathematical expectation operator. Furthermore, proceeding identically  as  in \cite{Peng1999PTRF} and \cite{Jiang2008AAP}, from (i) of \cref{thm:3.3} and (i) of \cref{thm:3.12} together with \eqref{eq:4.1} we easily (thus omitting the proof) have the following two propositions.\vspace{0.1cm}

\begin{pro}\label{pro:4.2}
$\ecal_g[\cdot]$ possesses the following properties:
\begin{itemize}
\item [(i)] {\bf Preserving  of constants}: For each constant $c\in\R$, $\ecal_g[c]=c$;
\item [(ii)] {\bf Monotonicity}: For each $\xi_1,\xi_2\in L^1(\F_T)$, if $\xi_1\geq \xi_2$, then $\ecal_g[\xi_1]\geq \ecal_g[\xi_2]$.
\end{itemize}
\end{pro}

\begin{pro}\label{pro:4.3}
For each $t\in\T$, $\ecal_g[\cdot|\F_t]$ possesses the following properties:
\begin{itemize}
\item [(i)] If $\xi\in L^1(\F_t)$, then $\ecal_g[\xi|\F_t]=\xi$;
\item [(ii)] For each $\xi_1,\xi_2\in L^1(\F_T)$, if $\xi_1\geq \xi_2$, then $\ecal_g[\xi_1|\F_t]\geq \ecal_g[\xi_2|\F_t]$;
\item [(iii)] For each $\xi\in L^1(\F_T)$ and $r\in [0,T]$, we have $\ecal_g[\ecal_g[\xi|\F_t]|\F_r]=\ecal_g[\xi|\F_{t\wedge r}]$;
\item [(iv)] For each $A\in \F_t$ and $\xi\in L^1(\F_T)$, $\ecal_g[{\bf 1}_A\xi|\F_t]={\bf 1}_A\ecal_g[\xi|\F_t]$ and $\ecal_g[{\bf 1}_A\xi]=\ecal_g[{\bf 1}_A\ecal_g[\xi|\F_t]]$.\vspace{0.1cm}
\end{itemize}
\end{pro}

It can be indicated from both propositions that the (conditional) $g$-expectation preserves essential properties (but except linearity) of the classical expectations. Some extensive issues on the (conditional) $g$-expectation still remain to be further studied along the lines \vspace{0.1cm} of~\cite{Peng1997Book,Chen1998CRAS,Peng1999PTRF,CoquetHuMeminPeng2002PTRF,
Peng2004AMAS,Peng2004CRAS,Peng2004LectureNotes,Jiang2008AAP,Delbaen2009MF}.

\begin{rmk}\label{rmk:4.3*}
In the same way as above, by \cref{thm:3.3,thm:3.6,thm:3.12,thm:3.13} one can define the (conditional) $g$-expectation via the solutions of BSDEs on the spaces
$$L(\ln L)^p\ (p>0),\ \ \bigcap\limits_{\mu>0}L\exp[\mu(\ln L)^p]\ (p>0),\ \ L^p\ (p>1), \ \ \bigcap\limits_{\mu>0} \exp(\mu L^{2\over \alpha^*})\  (\alpha^*\geq 2), $$
and $L^\infty$,
respectively. It is clear that the generator $g$ of BSDEs needs to satisfy some stronger conditions as the space becomes larger. In particular, when $g(t,y,z):\equiv \gamma |z|$, the corresponding conditional $g$-expectation $\ecal_g[\cdot|\F_t]$ for $t\in \T$ can be defined on the space $\bigcap_{\mu>0}L\exp[\mu\sqrt{\ln L}]$, which is bigger than $L^p\ (p>1)$ used for example in \cite{Peng1997Book,Chen1998CRAS,CoquetHuMeminPeng2002PTRF,Peng2004AMAS,
Peng2004CRAS,Peng2004LectureNotes,Tang2006CRA,Jia2008PHDThesis,Jia2010SPA}. Furthermore, according to (iii) of \cref{thm:3.3} and (iii) of \cref{thm:3.12}, we can verify that for each $t\in \T$ and $\xi\in \bigcap_{\mu>0}L\exp[\mu\sqrt{\ln L}]$,
$$
\ecal_g[\xi|\F_t]:=\esup\limits_{q\in \acal} \E_q[\xi|F_t]
$$
with $\acal$ being defined in \cref{ex:2.1}. This is just the maximal conditional expectation on $\acal$.\vspace{0.2cm}
\end{rmk}

\subsection{Dynamic utility process and risk measure\vspace{0.2cm}}

In the sequel, we introduce an application of our theoretical results in mathematical finance. For simplicity of  notations, we set for each $t\in \T$ and $\alpha\in (1,2]$
$$
E^\alpha(\F_t):=\bigcap_{\mu>0} \exp[\mu L^{2\over \alpha^*}](\F_t)
$$
with $\alpha^*:={\alpha\over \alpha-1}\geq 2$ being the conjugate of $\alpha$. Clearly, $E^\alpha(\F_t)$ is a linear space containing $L^\infty(\F_t)$ of bounded random variables. The following proposition is a direct consequence of (iii) of \cref{thm:3.12}, and the proof is omitted.

\begin{pro}\label{pro:4.4}
Suppose that the generator $g(z):\R^d\To \R$ is a concave function  satisfying $g(0)=0$ and
\begin{equation}\label{eq:4.3}
|g(z)|\leq a+\gamma |z|^\alpha
\end{equation}
with $a\geq0$ and $\alpha\in (1,2]$ being two given constants. Then, for each $\xi\in E^\alpha(\F_T)$, BSDE$(\xi,g)$ admits a unique solution $(Y_t,Z_t)_{t\in\T}$ such that $Y_t\in E^\alpha(\F_t)$ for each $t\in \T$.\vspace{0.1cm}
\end{pro}

Now, by virtue of \cref{pro:4.4}, for each $\xi\in E^\alpha(\F_T)$ we can define
\begin{equation}\label{eq:4.4}
U_t^g(\xi):=Y_t^\xi,\ \ t\in \T,
\end{equation}
where $(Y_t^\xi,Z_t^\xi)_{t\in\T}$ is the unique solution of BSDE$(\xi,g)$ such that $Y_t^\xi\in E^\alpha(\F_t)$ for each $t\in \T$.

The following theorem indicates that the family of operators $\{U_t^g(\cdot)\}_{t\in \T}$ defined via \eqref{eq:4.4} constitutes a dynamic utility process defined on $E^\alpha(\F_T)$.

\begin{thm}\label{thm:4.5}
For each $t\in\T$, the mapping $U_t^g(\cdot):E^\alpha(\F_T)\To E^\alpha(\F_t)$ defined via \eqref{eq:4.4} satisfies the following properties:

(i) {\bf Positivity}: $U_t^g(0)=0$ and $U_t^g(\xi)\geq 0$ for each nonnegative random variable $\xi\in E^\alpha(\F_T)$;

(ii) {\bf Monotonicity}: for each $\xi,\eta\in E^\alpha(\F_T)$, if $\xi\geq \eta$, then $U_t^g(\xi)\geq U_t^g(\eta)$;

(iii) {\bf Monetary}: $U_t^g(\xi+\eta)=U_t^g(\xi)+\eta$ for each $\xi\in E^\alpha(\F_T)$ and $\eta\in E^\alpha(\F_t)$;

(iv) {\bf Concavity}: $U_t^g(\theta\xi+(1-\theta)\eta)\geq \theta U_t^g(\xi)+(1-\theta)U_t^g(\eta)$ for all $\xi,\eta\in E^\alpha(\F_T)$ and $\theta\in (0,1)$.
\end{thm}

\begin{proof}
In view of $g(0)=0$ and the fact that $g$ is independent of $y$, (i)-(iii) are the direct consequences of (iii) of \cref{thm:3.12} and (i) of \cref{thm:3.6}. Furthermore, proceeding identically as Proposition 3.5 in \cite{ElKarouiPengQuenez1997MF}, by virtue of (i) of \cref{thm:3.6} and the concavity of $g$ we can get (iv).
\end{proof}

Now, we let the function $f:\R^d\To \R_+$ be convex, satisfy $f(0)=0$, and $\liminf_{|x|\To \infty} f(x)/|x|^2>0$. For each $\xi\in L^\infty(\F_T)$, define
\begin{equation}\label{eq:4.5}
\bar U_t(\xi):={\rm essinf}\left\{\left.\E_q\left[\left. \xi+\int_t^T f(q_u){\rm d}u\right|\F_t\right]\right| \Q^q\sim\p \right\},\ \ t\in \T,
\end{equation}
where $\E_q[\cdot|\F_t]$ is the conditional expectation operator with respect to $\F_t$ under the probability measure $\Q^q$, which is equivalent to $\p$ and
$$
\E\left[\left.\frac{{\rm d}\Q^q}{{\rm d}\p}\right|\F_t\right]
=\exp\left\{\int_0^t q_u\cdot {\rm d}B_u-\frac{1}{2}\int_0^t |q_u|^2{\rm d}u\right\},\ \ t\in\T.
$$
It is not difficult to check that $\{\bar U_t(\cdot)\}_{t\in \T}$ defined via \eqref{eq:4.5} constitutes a dynamic utility process defined on $L^\infty(\F_T)$. And, it follows from Theorems 2.1-2.2 in \cite{DelbaenHuBao2011PTRF} that there exists a $(Z_t)_{t\in\T}\in \mcal^2$ such that $(\bar U_t(\xi), Z_t)_{t\in\T}$ is the unique bounded solution of the following BSDE
\begin{equation}\label{eq:2-54}
Y_t=\xi+\int_t^T g(Z_s){\rm d}s-\int_t^T Z_s\cdot {\rm d}B_s,\ \ t\in \T,
\end{equation}
where
\begin{equation}\label{eq:2-55}
g(z):=-\sup_{x\in \R^d}(-z\cdot x-f(x))=\inf_{x\in \R^d}(z\cdot x+f(x))\leq 0,\ \ \RE z\in \R^d\vspace{0.1cm}
\end{equation}
is a concave function with $g(0)=0$, and $\limsup_{|x|\To \infty} |g(x)|/|x|^2<\infty$. For example, if $c>0$ and
$$
f(x)=c|x|^{\alpha^*}\geq 0,\ \ x\in \R^d,
$$
then
$$
0\geq g(z)=-\frac{1}{c^{\alpha-1}}\frac{\alpha^*-1}{(\alpha^*)^\alpha}|z|^\alpha,\ \ z\in \R^d,\vspace{0.1cm}
$$
which means that $g$ satisfies the conditions in \cref{pro:4.4}.

\begin{rmk}\label{rmk:4.6}
It is well known that \eqref{eq:4.5} is usually used to define the utility of bounded endowments in mathematical finance, see for example \cite{DelbaenHuBao2011PTRF} for details. However, it is only defined on the space $L^\infty(\F_T)$. This motivates the definition \eqref{eq:4.4} via a BSDE, which can be defined on a larger space $E^\alpha(\F_T)$ than $L^\infty(\F_T)$. Some relevant results are available in \cite{ElKarouiPengQuenez1997MF,DelbaenPengRosazza2010FS,FanHuTang2023SPA}. In a symmetric way to the above, we can also define a convex risk measure on $E^\alpha(\F_T)$, see \cite{FollmerSchied2002FS,Jiang2008AAP} among others for more details.
\end{rmk}

\begin{ex}\label{ex:4.8}
Let the generator
$$
g(z):=c({\bf 1}_{\lambda\geq 0}-{\bf 1}_{\lambda<0})|z|(\ln (e+|z|))^\lambda,\ \ \ z\in \R^d,
$$
where $c>0$ and $\lambda\in \R$ are two given constants. From~\cref{thm:3.12}, we  easily have the following three assertions:
\begin{itemize}
\item [(i)] If $\lambda\in (-\infty,-\frac{1}{2})$, then for each $\xi\in L^1$, BSDE$(\xi,g)$ admits a unique solution $(Y_t^\xi,Z_t^\xi)_{t\in \T}$ such that $(Y_t^\xi)_{t\in\T}$ is of class (D);

\item [(ii)] If $\lambda=-\frac{1}{2}$, then for each $\xi\in L(\ln L)^p$ with $p>0$, BSDE$(\xi,g)$ admits a unique solution $(Y_t^\xi,Z_t^\xi)_{t\in \T}$ such that the process
$$|Y_t^\xi|(\ln (e+|Y_t^\xi|))^p, \quad t\in\T$$
 is of class (D);

\item [(iii)] If $\lambda\in (-\frac{1}{2},+\infty)$ and $p:=(\lambda+{1\over 2})\vee (2\lambda)$, then for each
$$\xi\in \cap_{\mu>0} L\exp[\mu(\ln L)^p],$$ BSDE$(\xi,g)$ admits a unique solution $(Y_t^\xi,Z_t^\xi)_{t\in \T}$ such that the process $(|Y_t^\xi|\exp (\mu(\ln (e+|Y_t^\xi|))^p))_{t\in\T}$ is of class (D) for each $\mu>0$.
\end{itemize}

\noindent Thus, for the preceding three different  ranges of $\lambda$, we can define the following operator
$$
\varrho(\xi):=Y_0^{-\xi}
$$
in three different spaces of contingent claims: $L^1$, $L(\ln L)^p$ and $\cap_{\mu>0} L\exp[\mu(\ln L)^p]$. Moreover, according to the properties of the generator $g$ together with \cref{thm:3.3,thm:3.6,thm:3.12}, in the same spirit as in for example \cite{Artzner1999MF,FollmerSchied2002FS,Jiang2008AAP,CheriditoLi2009MF,
Delbaen2009MF,Castagnoli2022OR}, we verify that whenever $\lambda>0$, $\lambda<0$ and $\lambda=0$,  $\varrho(\cdot)$ is  respectively a convex risk measure, a star-shaped risk measure (see the precise definition in page 2641 of \cite[Definition 1]{Castagnoli2022OR}) and a coherent risk measure on the corresponding space of contingent claims. In addition, in the same way,  we can also define the corresponding dynamic risk measure with the solution $Y_t^{-\xi}$ of BSDE$(\xi,g)$ for $t\in \T$.
\end{ex}

\subsection{Nonlinear Feynman-Kac formula\vspace{0.2cm}}

As another application of our theoretical results, in this subsection we will derive a nonlinear Feynman-Kac formula for PDEs which are at most quadratic with respect to the gradient of the solution. Let us consider the following semi-linear PDE:
\begin{equation}\label{eq:4.8}
\partial_t u(t,x)+\mathcal{L} u(t,x)+g(t,x,u(t,x),\sigma^* \nabla_x u(t,x))=0,\ \ \ u(T,\cdot)=h(\cdot),
\end{equation}
where $\mathcal{L}$ is the infinitesimal generator of the solution $X^{t,x}_\cdot$ to the following SDE:
\begin{equation}\label{eq:4.9}
X^{t,x}_s=x+\int_t^s b(r,X^{t,x}_r){\rm d}r+\int_t^s \sigma(r,X^{t,x}_r){\rm d}B_r,\quad t\leq s\leq T.
\end{equation}
For each $(t_0,x_0)\in \T\times \R^n$, $(Y^{t_0,x_0}_t,Z^{t_0,x_0}_t)_{t\in [t_0,T]}$ is the solution to the BSDE
\begin{equation}\label{eq:4.11}
Y_t=h\left(X^{t_0,x_0}_T\right)+\int_t^T g(s,X^{t_0,x_0}_s,Y_s,Z_s){\rm d}s+\int_t^T Z_s\cdot {\rm d}B_s,\quad t\in [t_0,T],
\end{equation}
The nonlinear Feynman-Kac formula says that the function
\begin{equation}\label{eq:4.10}
u(t,x):=Y^{t,x}_t, \quad \RE\ (t,x)\in \T\times \R^n,
\end{equation}
is a viscosity solution to PDE \eqref{eq:4.8}.

Let us first recall the definition of a continuous viscosity solution in our framework, see e.g. \cite{CrandallIshiiLions1992BAMS}.

\begin{defn}\label{def:4.7}
A continuous function $u:\T\times \R^n\To \R$ with $u(T,\cdot)=h(\cdot)$ is said to be a viscosity super-solution (resp. sub-solution) to PDE \eqref{eq:4.8} if the inequality
$$
\partial_t u(t_0,x_0)+\mathcal{L} u(t_0,x_0)+g(t_0,x_0,u(t_0,x_0),\sigma^* \nabla_x \varphi(t_0,x_0))\leq 0\ \ \ ({\rm resp.}\ \ \geq 0)
$$
holds true for any smooth function $\varphi(\cdot,\cdot)$ such that  the function $u-\varphi$ attains a local minimum (resp. maximum) at the point $(t_0,x_0)\in (0,T)\times \R^n$. Moreover, a viscosity super-solution is said to be a viscosity solution if it is also a viscosity sub-solution.
\end{defn}

Let us now introduce the following assumptions on the coefficients of SDE \eqref{eq:4.9}.\vspace{0.1cm}

\begin{enumerate}
\renewcommand{\theenumi}{(A1)}
\renewcommand{\labelenumi}{\theenumi}
\item\label{A1}
Both functions $b:\T\times \R^n\To \R^n$ and $\sigma:\T\times \R^n\To \R^{n\times d}$ are jointly continuous and there is a positive constant $K>0$ such that for each $(t,x,x')\in \T\times\R^n\times\R^n$,
$$
|b(t,0)|+|\sigma(t,x)|\leq K
$$
and
$$
|b(t,x)-b(t,x')|+|\sigma(t,x)-\sigma(t,x')|\leq K|x-x'|.\vspace{0.2cm}
$$
\end{enumerate}

Classical results on SDEs show that under the assumption \ref{A1}, SDE \eqref{eq:4.9} has a unique solution $X^{t,x}_\cdot\in \cap_{q\geq 1}\s^q$ for each $(t,x)\in \T\times \R^n$. And, since $\sigma$ is bounded, the argument in page 563 of Briand and Hu~\cite{BriandHu2008PTRF} yields that for each $\mu>0$, there is a constant $C>0$,  depending only on $(q,\mu,T,K)$,  such that for each $q\in [1,2)$ and $(t,x)\in \T\times \R^n$,
\begin{equation}\label{eq:4.12}
\E\left[\sup_{s\in [t,T]} \exp\left(\mu |X^{t,x}_s|^q\right)\right]\leq C \exp(\mu C |x|^q).
\end{equation}

Let us further give our assumptions on the generator $g$ and the terminal condition of BSDE \eqref{eq:4.11}.

\medskip
\begin{enumerate}
\renewcommand{\theenumi}{(A2)}
\renewcommand{\labelenumi}{\theenumi}
\item\label{A2} Both functions $g:\T\times \R^n\times\R\times\R^d\To \R$ and $h:\R^n\To \R$ are jointly continuous and there are three real constants $k\geq 0$, $\alpha\in (1,2]$ and $p\in [1,\alpha^*)$ with $\alpha^*$ being the conjugate of $\alpha$ such that we have for each $(t,x,y,z)\in \T\times \R^n\times \R\times\R^d$,
\begin{equation}\label{eq:4.13}
{\rm sgn}(y)g(t,x,y,z)\leq k \left(1+|x|^p+|y|+|z|^\alpha\right),
\end{equation}
\begin{equation}\label{eq:4.14}
|g(t,x,y,z)|+|h(x)|\leq k\left(1+|x|^p+\exp(k |y|^{2\over \alpha^*})+|z|^2\right),
\end{equation}
and either inequality
$$
  \begin{array}{l}
  \Dis {\bf 1}_{y-\theta y'>0}\left(g(t,x,y,z)-\theta g(t,x, y',z')\right)\\[3mm]
  \ \ \leq \Dis (1-\theta) k \left(1+|x|^p+|y'|+\left(\frac{y-\theta y'}{1-\theta}\right)^+ +\left|\frac{z-\theta z'}{1-\theta}\right|^\alpha\right)
  \end{array}
$$
   or
$$
  \begin{array}{l}
  \Dis -{\bf 1}_{y-\theta y'<0}\left(g(t,x,y,z)-\theta g(t,x, y',z')\right)\\[3mm]
  \ \ \leq \Dis (1-\theta) k\left(1+|x|^p+|y'|+\left(\frac{y-\theta y'}{1-\theta}\right)^- +\left|\frac{z-\theta z'}{1-\theta}\right|^\alpha\right)\vspace{0.3cm}
   \end{array}
$$
 is satisfied for any $(t,x,y,y',z,z')\in \T\times \R^n\times \R^2\times\R^{2d}$ and $\theta\in (0,1)$.
\end{enumerate}

\medskip
Since $p\in [1,\alpha^*)$,  we have from \eqref{eq:4.12} that for each $(t_0,x_0)\in \T\times\R^n$ and each $\mu>0$,
$$
\begin{array}{lll}
\Dis \E\left[\exp\left(\mu \left(|X^{t_0,x_0}_T|^p\right)^{2\over \alpha^*}\right)\right] &\leq& \Dis  \E\left[\sup_{s\in [t_0,T]} \exp\left(\mu |X^{t_0,x_0}_s|^{2p\over \alpha^*}\right)\right]\vspace{0.2cm}\\
&\leq&  \Dis \bar C \exp(\mu \bar C |x_0|^{2p\over \alpha^*})<+\infty
\end{array}
$$
and
$$
\E\left[\exp\left(\mu \left(\int_{t_0}^T |X^{t_0,x_0}_s|^p{\rm d}s\right)^{2\over \alpha^*}\right)\right] \leq\Dis  \E\left[\sup_{s\in [t_0,T]} \exp\left(\mu T^{2\over \alpha^*} |X^{t_0,x_0}_s|^{2p\over \alpha^*}\right)\right]<+\infty,\vspace{0.1cm}
$$
for a constant $\bar C>0$  depending only on $(p,\alpha,\mu,T,K)$. Then, in view of the assumption \ref{A2} and the last two inequalities, we can apply \cref{thm:3.13} to construct a unique solution $(Y^{t_0,x_0}_t,Z^{t_0,x_0}_t)_{t\in [t_0,T]}$ to BSDE \eqref{eq:4.11} such that $(\exp(\mu |Y^{t_0,x_0}_t|^{2\over \alpha^*}) )_{t\in [t_0,T]}$ is of class (D) for each $\mu>0$. Furthermore, a classical argument yields that the function $u$ defined by~\eqref{eq:4.10} is deterministic.

The following theorem constitutes the main result of this subsection.

\begin{thm}\label{thm:4.8}
Let assumptions \ref{A1} and \ref{A2} be satisfied. Then, the function $u$ defined in \eqref{eq:4.10} is continuous on $\T\times\R^n$ and there exists a constant $C>0$ such that
$$
|u(t,x)|\leq C(1+|x|^p), \quad \RE\ (t,x)\in \T\times\R^n.
$$
Moreover, $u$ is a viscosity solution to PDE \eqref{eq:4.8}.
\end{thm}

The proof is available in \cite{FanHu2021SPA,BriandHu2008PTRF}. Similar results can also be found in \cite{DalioLey2006SCON,DalioLey2011AMO}.

\begin{rmk}\label{rmk:4.11}
The nonlinear Feynman-Kac formula for solution of PDEs can be dated back to  \cite{Peng1991Stochastics}, in the spirit of which numerical discussions appeared  successively in for example \cite{Peng1992Stochastics,Peng1993AMO,Pardoux1999Nonlinear,PardouxTang1999PTRF,
Kobylanski2000AP,Jia2008PHDThesis,BriandHu2008PTRF,PardouxRascanu2014Book,
FanHu2021SPA}. In fact, according to
\cref{thm:3.3,thm:3.6,thm:3.12,thm:3.13}, we can establish a one-to-one correspondence between solutions of PDEs and BSDEs via the Feynman-Kac formulas, as mentioned in section 4 of \cite{Jia2008PHDThesis}. In general, the generator $g$ and the terminal condition $h$ of BSDE \eqref{eq:4.11} can admit a more general growth  in the unknown variable $x$ when $g$ has a lower growth  in the unknown variable $z$.
\end{rmk}

\section{Open problems}
\label{sec:5-Problems}
\setcounter{equation}{0}

In this section, we  describe five open problems. The first two concern the existence of a solution to a BSDE and a PDE,  and the last three  address  the uniqueness.\vspace{0.2cm}

{\bf Problem 5.1.} Consider the following BSDE:
\begin{equation}\label{eq:5.1}
Y_t=\xi+\int_t^T \frac{|Z_s|}{\sqrt{\ln(e+|Z_s|)}} {\rm d}s-\int_t^T Z_s\cdot {\rm d}B_s,\ \ t\in\T.
\end{equation}
For  $\xi\in L^1$, does BSDE \eqref{eq:5.1} admit a solution $(Y_t,Z_t)_{t\in\T}$ such that $Y$ is of class (D)? Note that  Assertion (ii) of \cref{thm:2.3} concerns the existence  of an adapted  solution of the BSDE for $\xi\in L(\ln L)^p$ only when  $p>0$, while Assertion  (i) of \cref{thm:2.3} concerns the BSDE of a more slowly growing   generator.\vspace{0.2cm}

{\bf Problem 5.2.} Let  assumption~\ref{A1} be satisfied and that the functions $g$ and $h$ only satisfy \eqref{eq:4.13} and \eqref{eq:4.14} of assumption~\ref{A2}  in the last section. Does the semi-linear PDE \eqref{eq:4.8} admit a viscosity solution? Note that  the PDE is associated to a  BSDE~\eqref{eq:4.11}, which  has (from~\cref{thm:2.6}) a solution $(Y^{t_0,x_0}_t,Z^{t_0,x_0}_t)_{t\in [t_0,T]}$ such that
    $(\exp(\mu |Y^{t_0,x_0}_t|^{2\over \alpha^*}) )_{t\in [t_0,T]}$ is of class (D) for each $\mu>0$.\vspace{0.2cm}

{\bf Problem 5.3.} Let $p>1$ and consider the following BSDE:
\begin{equation}\label{eq:5.6}
Y_t=\xi+\int_t^T |Z_s|\sin|Z_s| {\rm d}s-\int_t^T Z_s\cdot {\rm d}B_s,\ \ t\in\T.
\end{equation}
For each $\xi\in L^p$, is the solution $(Y_t,Z_t)_{t\in\T}$ to the last BSDE with $(|Y_t|^p)_{t\in\T}$ being of class (D) unique?
We note that the generator $g$ is not uniformly continuous in $z$, while the uniqueness assertion (iv) of \cref{thm:3.3} requires the uniform continuity of  the generator $g$ in $z$.\vspace{0.2cm}

{\bf Problem 5.4.} Consider the following BSDE:
\begin{equation}\label{eq:5.4}
Y_t=\xi+\int_t^T g(Z_s)\  {\rm d}s-\int_t^T Z_s\cdot {\rm d}B_s,\ \ t\in\T,
\end{equation}
where $g$ is only convex (rather than  strongly convex)  and satisfies $0\le g(z)\le \frac12|z|^2$. We recall that $g$ is said to be strongly convex if there exists
two constants $C_1>0$ and $C_2\ge 0$ such that $\forall z,z', \forall s\in \partial g(z')$ ($\partial g$ denotes the subdifferential of $g$)
$$g(z)-g(z')-s\cdot (z-z')\ge C_1| z-z'|^2-C_2.$$
For each $\xi\in \exp(L)$, is the solution $(Y_t,Z_t)_{t\in\T}$ to the last BSDE with $(\exp(|Y_t|))_{t\in\T}$ being of class (D) unique?
Whenever  the quadratic generator $g$ is strongly convex  in $z$, Theorem 4.1 of~\cite{DelbaenHuRichou2015DCDS} gives the uniqueness of the solution to BSDE$(\xi,g)$  for $\xi\in \exp(L)$. While for a terminal value $\xi\in \exp(pL)$ with $p>1$, Theorem 3.3 of~\cite{DelbaenHuRichou2011AIHPPS}  gives the uniqueness of  the solution $(Y_t,Z_t)_{t\in\T}$ to BSDE \eqref{eq:5.4} such that $(\exp(p|Y_t|))_{t\in\T}$ is of class (D).\vspace{0.2cm}

{\bf Problem 5.5.}  Consider the following BSDE:
\begin{equation}\label{eq:5.5}
Y_t=\xi+\int_t^T g(s,Y_s, Z_s)\  {\rm d}s-\int_t^T Z_s\cdot {\rm d}B_s,\ \ t\in\T,
\end{equation}
where  the generator $g$ is neither convex nor concave in $z$, but satisfies \eqref{eq:3.12*} . For each $\xi\in \cap_{\mu>0}\exp(\mu L)$, is the solution $(Y_t,Z_t)_{t\in\T}$ to BSDE \eqref{eq:5.5} with $(\exp(\mu |Y_t|))_{t\in\T}$ being of class (D) for each $\mu>0$ unique?  When $\xi\in L^\infty$, it is well known from ~\cite{Kobylanski2000AP} that the solution $(Y_t,Z_t)_{t\in\T}$ to BSDE \eqref{eq:5.5} with $Y\in \s^\infty$ is unique.   For $\xi\in \cap_{\mu>0}\exp(\mu L)$,  if the quadratic generator $g$ has a strictly positive (or negative) quadratic growth and an extended convexity (or concavity) in $z$,   Theorem 5 together with Remark 7 of~\cite{FanHuTang2020CRM} already give the uniqueness of the solution $(Y_t,Z_t)_{t\in\T}$ to BSDE$(\xi,g)$ such that  $(\exp(\mu |Y_t|))_{t\in\T}$ is of class (D) for each $\mu>0$.\vspace{0.2cm}

Finally,  we  note that the localization method, Girsanov's transform, the comparison theorem and the $\theta$-difference technique, successfully used for one-dimensional BSDEs, do not apply to general multidimensional BSDEs. The counterexample by Frei and Dos Reis~\cite{FreiDosReis2011MFE} shows that  solving a multidimensional quadratic BSDE with a bounded terminal value, without a structural assumption, is impossible. In 1999, Peng~\cite{Peng1999Open} listed ``the solvability of multidimensional quadratic BSDEs with unbounded terminal values" as an open problem, which is only partially solved.  Although the comparison theorem for adapted solutions of multidimensional BSDEs was established and applied (see for example \cite{HuPeng2006CRAS,HuTang2010PTRF,HuShiXu2025SICON}) and Girsanov's transform (see for example \cite{HuTang2016SPA,FanHuTang2023JDE,FanHuTang2025EJP,
HaoHuWen2025AAP}) and the $\theta$-difference technique (see for example \cite{FanHuTang2023JDE,HuTangWang2022SPA}) were also applied to multidimensional BSDEs, all these works assume specific structural conditions.

\appendix
\section{Proof of \cref{pro:2.5}}
\renewcommand{\appendixname}{}

\begin{proof}[Proof of \cref{pro:2.5}] The case of $\lambda=0$ is clear. Let us consider the case of $\lambda\neq 0$. Given $k\geq k_{\lambda,p}$ with $\bar p:=p^{1\over 2|\lambda|}>1$, $\bar\lambda:=2|\lambda|(2+|\lambda-1|)$, $k_{\lambda,p}\geq e^{|\lambda-1|+1}$,
\begin{equation}\label{eq:2.8+}
\bar\lambda \left(\ln k_{\lambda,p}\right)^{\lambda-1}< k_{\lambda,p}^{1-\frac{1}{\bar p}},\ \ k_{\lambda,p}^{\bar p}-k_{\lambda,p}-2\sqrt{p} k_{\lambda,p}(\ln k_{\lambda,p})^\lambda>0
\end{equation}
and
$$
k_{\lambda,p}^{1\over \bar p}<\sqrt{p} k_{\lambda,p}(\ln k_{\lambda,p})^\lambda.\vspace{0.1cm}
$$
For each $(x,y)\in (0,+\infty)\times (0,+\infty)$, define the function
\begin{equation}\label{eq:A-5}
\begin{array}{l}
f(x,y):=\Dis y^2-2xy\left(\ln (k+y)\right)^\lambda +px^2\left(\ln (k+x)\right)^{2\lambda}\\
\hspace{1.2cm}=\Dis \left(y- x\left(\ln (k+y)\right)^\lambda\right)^2+px^2\left(\ln (k+x)\right)^{2\lambda}
-x^2\left(\ln (k+y)\right)^{2\lambda}.
\end{array}
\end{equation}
Clearly, in order to prove \eqref{eq:2.8}, it is enough to prove that $f(x,y)\geq 0$ for each $x,y>0$. Fix arbitrarily $x\in (0,+\infty)$ and let $\bar f(y):=f(x,y),\ y\in (0,+\infty)$. A simple calculation gives that for each $y\in (0,+\infty)$,
\begin{equation}\label{eq:A-6}
\begin{array}{lll}
\bar f'(y)&=&\Dis 2y-2x\left(y\left(\ln (k+y)\right)^\lambda\right)'\vspace{0.1cm}\\
&=& \Dis 2y-2x(\ln (k+y))^\lambda\left[1+\frac{\lambda y}{(k+y)\ln (k+y)}\right]
\end{array}
\end{equation}
and
\begin{equation}\label{eq:A-7}
\begin{array}{lll}
\bar f''(y)&=& \Dis 2-2x \left(y\left(\ln (k+y)\right)^\lambda\right)''\vspace{0.1cm}\\
&=&\Dis 2-\frac{2\lambda x[(2k+y)\ln(k+y)-(1-\lambda)y]}{(k+y)^2\left(\ln (k+y)\right)^{2-\lambda}}.\vspace{0.1cm}
\end{array}
\end{equation}
Furthermore, let $y_0\in \R$ be the unique constant depending only on $(p,k,\lambda,x)$ and satisfying
\begin{equation}\label{eq:A-8}
p\left(\ln (k+x)\right)^{2\lambda}=\left(\ln (k+y_0)\right)^{2\lambda}
\end{equation}
or equivalently,
$$x=(k+y_0)^{\frac{1}{p^{1\over 2\lambda}}}-k\ \ {\rm or}\ \ y_0=(k+x)^{p^{1\over 2\lambda}}-k.
$$
It then follows from \eqref{eq:A-5} that $\bar f(y_0)=f(x,y_0)\geq 0$.

In the sequel, we will distinguish two different cases to prove the desired inequality \eqref{eq:2.8}.

{\bf Case 1: $\lambda>0$}. In this case, $y_0$ is a positive number. By \eqref{eq:A-5} and \eqref{eq:A-8} we know that
$$
\begin{array}{lll}
\bar f(y)&\geq &\Dis  px^2\left(\ln (k+x)\right)^{2\lambda}
-x^2\left(\ln (k+y)\right)^{2\lambda}\\
&\geq &\Dis  x^2 \left[p\left(\ln (k+x)\right)^{2\lambda}
-\left(\ln (k+y_0)\right)^{2\lambda}\right]=0,\ \ y\in (0,y_0].
\end{array}
$$
Hence, it suffices to verify that $\bar f(y)\geq 0$ for $y\in [y_0,+\infty)$. In fact, by \eqref{eq:A-7}, \eqref{eq:A-8} and \eqref{eq:2.8+} we have
$$
\begin{array}{lll}
\bar f''(y)&\geq & \Dis 2-\frac{2|\lambda| x[2(k+y)\ln(k+y)+|\lambda-1|(k+y)\ln(k+y)]}{(k+y)^2\left(\ln (k+y)\right)^{2-\lambda}}\vspace{0.2cm}\\
&= & \Dis 2-\frac{\bar\lambda\left(\ln (k+y)\right)^{\lambda-1}}{k+y}\left((k+y_0)^{1\over \bar p}-k\right)
\vspace{0.2cm}\\
&\geq & \Dis 2-\frac{\bar\lambda\left(\ln (k+y_0)\right)^{\lambda-1}}{(k+y_0)^{1-{1\over \bar p}}}>2-\frac{\bar\lambda (\ln k)^{\lambda-1}}{k^{1-{1\over \bar p}}}>1,\ \ y\in [y_0,+\infty).
\end{array}
$$
And, by \eqref{eq:A-6}, \eqref{eq:A-8} and \eqref{eq:2.8+} we can deduce that, in view of $\bar p>1$,
\begin{equation}
\begin{array}{lll}
\bar f'(y_0)&\geq & \Dis 2y_0-2x(\ln (k+y_0))^\lambda\left[1+\frac{\lambda }{\ln (k+y_0)}\right]\vspace{0.1cm}\\
&\geq & \Dis 2y_0-4x(\ln (k+y_0))^\lambda\vspace{0.1cm}\\
&=& \Dis 2(k+x)^{\bar p}-2k-4\sqrt{p} x(\ln (k+x))^\lambda\vspace{0.1cm}\\
&\geq & \Dis 2(k+x)^{\bar p}-2(k+x)-4\sqrt{p} (k+x)(\ln (k+x))^\lambda\vspace{0.1cm}\\
&\geq & \Dis 2k^{\bar p}-2k-4\sqrt{p} k(\ln k)^\lambda>0.
\end{array}
\end{equation}
Consequently, for each $y\in [y_0,+\infty)$, we have $\bar f'(y)\geq \bar f'(y_0)>0$ and then $\bar f(y)\geq \bar f(y_0)\geq 0$.\vspace{0.2cm}

{\bf Case 2: $\lambda<0$}. In this case, by \eqref{eq:A-5} and \eqref{eq:A-8} we know that
$$
\begin{array}{lll}
\bar f(y)&\geq &\Dis  px^2\left(\ln (k+x)\right)^{2\lambda}
-x^2\left(\ln (k+y)\right)^{2\lambda}\\
&\geq &\Dis  x^2 \left[p\left(\ln (k+x)\right)^{2\lambda}
-\left(\ln (k+y_0)\right)^{2\lambda}\right]=0,\ \ y\in [y_0^+,+\infty).
\end{array}
$$
Hence, it suffices to verify that $\bar f(y)\geq 0$ for $y_0>0$ and $y\in (0,y_0]$. In fact, by \eqref{eq:A-7} and \eqref{eq:2.8+} we have
$$
\begin{array}{lll}
\bar f''(y)&\geq& \Dis 2+\frac{2|\lambda| x\left[(k+y)\ln(k+y)-|1-\lambda|(k+y)\right]}{(k+y)^2\left(\ln (k+y)\right)^{2-\lambda}}\vspace{0.2cm}\\
&=& \Dis 2+\frac{2|\lambda| x \left[\ln (k+y)-|1-\lambda|\right]}{(k+y)\left(\ln (k+y)\right)^{2-\lambda}}>2,\ \ y\in (0,y_0].
\end{array}
$$
And, by \eqref{eq:A-6}, \eqref{eq:A-8} and \eqref{eq:2.8+} we can deduce that, in view of $\bar p>1$,
\begin{equation}
\begin{array}{lll}
\bar f'(y_0)&\leq & \Dis 2y_0-2x(\ln (k+y_0))^\lambda \vspace{0.1cm}\\
&=& \Dis 2(k+x)^{1\over \bar p}-2k-2\sqrt{p} x(\ln (k+x))^\lambda\vspace{0.1cm}\\
&\leq & \Dis 2(k+x)^{1\over \bar p}-2\sqrt{p} (k+x)(\ln (k+x))^\lambda\vspace{0.1cm}\\
&\leq & \Dis 2k^{1\over \bar p}-2\sqrt{p} k(\ln k)^\lambda<0.
\end{array}
\end{equation}
Consequently, for each $y\in (0,y_0]$, we have $\bar f'(y)\leq \bar f'(y_0)<0$ and then $\bar f(y)\geq \bar f(y_0)\geq 0$.\vspace{0.2cm}

In conclusion, \eqref{eq:2.8} holds. Finally, we verify that when $p=1$ and $\lambda\neq 0$, the constant $k$ such that \eqref{eq:2.8} holds does not exist. In fact, assume that \eqref{eq:2.8} holds for some $k\geq e$. Let $x,y>0$ satisfy
$$
y=x\left(\ln (k+y)\right)^\lambda.
$$
It is clear that $y>x$ for $\lambda>0$, and $y<x$ for $\lambda<0$. Then, in view of \eqref{eq:A-5},
$$
y^2-2xy\left(\ln (k+y)\right)^\lambda +x^2\left(\ln (k+x)\right)^{2\lambda}=x^2\left[\left(\ln (k+x)\right)^{2\lambda}
-\left(\ln (k+y)\right)^{2\lambda}\right]<0,
$$
which  immediately yields the desired assertion. The proof is then complete.
\end{proof}

\begin{rmk}\label{rmk:A.1}
The case of $\lambda<0$ in \cref{pro:2.5} has been established in Proposition 3.2 of \cite{FanHuTang2023SCL}. However, our proof  is more direct and simpler. The case of $\lambda>0$ in \cref{pro:2.5} can be compared to Proposition 3.2 of \cite{FanHuTang2023SPA}, where the constant $p$ appearing in \eqref{eq:2.8} is required to be strictly bigger than $4^{(\lambda-1)^+}$. From this point of view, \cref{pro:2.5} improves Proposition 3.2 in \cite{FanHuTang2023SPA} for the case of $\lambda>1$.\vspace{0.1cm}
\end{rmk}

\section*{Acknowledgment}
The authors would like to thank the anonymous referee for his/her careful reading and valuable suggestions.




\end{document}